\documentclass[a4paper,10pt]{amsart}

\usepackage{enumerate, amsmath, amsfonts, amssymb, amsthm, mathtools, thmtools, wasysym, graphics, graphicx, xcolor, frcursive,xparse,comment}
\usepackage[all]{xy}
\usepackage{hyperref}
\hypersetup{colorlinks=true, citecolor=darkblue, linkcolor=darkblue}
\definecolor{darkblue}{rgb}{0.0,0,0.7} 
\newcommand{\darkblue}{\color{darkblue}} 
\definecolor{darkred}{rgb}{0.7,0,0} 
\definecolor{lightgrey}{rgb}{0.7,0.7,0.7} 

\newtheorem{theorem}{Theorem}[section]
\newtheorem{proposition}[theorem]{Proposition}

\newtheorem{lemma}[theorem]{Lemma}

\theoremstyle{definition}
\newtheorem{definition}[theorem]{Definition}

\newtheorem{remark}[theorem]{Remark}

\newcommand{\projPerm}{{\mathfrak{P}}}
\newcommand{\projCamb}{{\mathfrak{C}}}
\newcommand{\sq}[1]{{\rm #1}} 

\newcommand{\Tri}{{\mathrm{T}}}

\newcommand{\wo}{{w_\circ}} 
\NewDocumentCommand\wom{O{m}}{{w_\circ^{#1}}} 

\newcommand{\F}{\mathcal{F}}
\newcommand{\V}{\mathcal{V}}

\newcommand{\sw}[2]{\sq{#1}(\sq{#2})} 
\NewDocumentCommand\cwo{O{c}}{\sw{\wo}{#1}} 
\NewDocumentCommand\cwotilde{O{c}}{{\widetilde{\sq{w}_\circ}}{(\sq{#1})}} 
\NewDocumentCommand\cwom{O{c}O{m}}{{\sq{w}_\circ^{#2}}{(\sq{#1})}} 

\DeclareMathOperator{\inv}{inv} 

\usepackage[colorinlistoftodos]{todonotes}

\newcommand{\defn}[1]{\emph{\darkblue #1}} 
\renewcommand{\paragraph}[1]{\bigskip\noindent\textbf{#1.}} 

\NewDocumentCommand\Catm{O{m}}{{{\sf Cat}^{(#1)}}}

\newcommand\restr[2]{{
  \left.\kern-\nulldelimiterspace 
  #1 
  \vphantom{\big|} 
  \right|_{#2} 
  }}

\newcommand{\Weak}{{{\sf Weak}}}
\NewDocumentCommand\Weakm{O{m}}{{{\sf Weak}^{(#1)}}}
\NewDocumentCommand\Wm{O{m}}{{W}^{(#1)}}

\newcommand{\Assoc}{{{\sf Asso}}}

\NewDocumentCommand\Assocm{O{m}}{{{\sf Asso}^{(#1)}}}
\NewDocumentCommand\DeltaAssocm{O{m}}{{{\sf Asso}^{(#1)}_\Delta}}
\NewDocumentCommand\NablaAssocm{O{m}}{{{\sf Asso}^{(#1)}_\nabla}}

\newcommand{\Asoc}{{{\sf Asso}}}

\newcommand{\Po}{{\mathcal{P}}}

\newcommand{\Camb}{{{\sf Camb}}}
\newcommand{\Perm}{{{\sf Perm}}}

\NewDocumentCommand\Cambm{O{m}}{{{\sf Camb}^{(#1)}}}
\NewDocumentCommand\Cambsortm{O{m}}{{{\sf Camb}_\Sort^{(#1)}}}
\NewDocumentCommand\Cambncm{O{m}}{{{\sf Camb}_\NC^{(#1)}}}
\NewDocumentCommand\Cambassocm{O{m}}{{{\sf Camb}_\Assoc^{(#1)}}}

\NewDocumentCommand\deltaNCm{O{m}}{{{\sf NC}_{\delta}^{(#1)}}}

\NewDocumentCommand\DeltaNCm{O{m}}{{{\sf NC}_{\Delta}^{(#1)}}}

\newcommand{\Red}{{{\sf Red}}}

\newcommand{\NC}{{{\sf NC}}}
\NewDocumentCommand\NCm{O{m}}{{{\sf NC}^{(#1)}}}

\newcommand{\Sort}{{{\sf Sort}}}
\NewDocumentCommand\Sortm{O{m}}{{{\sf Sort}^{(#1)}}}

\NewDocumentCommand\deltaSortm{O{m}}{{{\sf Sort}_{{\sf shard}}^{(#1)}}}

\newcommand{\cov}{{{\sf cov}}}

\usepackage{cleveref}
\usepackage[all]{xy}
\usepackage[T1]{fontenc}

\title{$W$-Associahedra are In-Your-Face}

\author[N.~Williams]{Nathan Williams}

\date{\today}

\begin{document}

\begin{abstract}
  \vspace*{-5pt}
  We use a projection argument to uniformly prove that $W$-permutahedra and $W$-associahedra have the property that if $v,v'$ are two vertices on the same face $f$, then any geodesic between $v$ and $v'$ does not leave $f$.  In type $A$, we show that our geometric projection recovers a slight modification of the combinatorial projection in~\cite{sleator1988rotation}. 
\end{abstract}

\maketitle

\section{Introduction}

The underlying graph of the Tamari lattice is the $1$-skeleton of a classical associahedron associated to the Coxeter group $A_n$.  This graph has vertices given by triangulations of a regular $(n+3)$-gon, with edges between two triangulations $T,T'$ if one can obtain $T'$ from $T$ by flipping a single diagonal.  In 1988, D.~Sleator, R.~Tarjan, and W.~Thurston used a projection argument to prove that if $T,T'$ share a common diagonal, then every shortest path between $T$ and $T'$ leaves this diagonal untouched~\cite{sleator1988rotation}.\footnote{C.~Ceballos and V.~Pilaud call this the ``non-leaving-face property'' in~\cite{ceballos2014diameter}.  We find our nomenclature a bit more on the nose.}  Using type-dependent combinatorial models for the $1$-skeleta of $W$-associahedra, this proof was recently extended by C.~Ceballos and V.~Pilaud to all classical types~\cite{ceballos2014diameter}.  C.~Ceballos and V.~Pilaud were able to check many exceptional types by computer, but were unable to finish the verification for $E_7$ (4160 vertices) and $E_8$ (25080 vertices).

In this note, we supply a short uniform proof of this result, drawing heavily on the beautiful theory of N.~Reading's Coxeter-sortable elements.  We first review the in-your-face property in Section~\ref{sec:inyourface}, summarizing the theory built in~\cite{sleator1988rotation,ceballos2014diameter}.  In Section~\ref{sec:permutahedron}, we then give a sesquipedalian proof that the $W$-permutahedron is in-your-face (Theorem~\ref{thm:perm_in_your_face}).  This provides the template for Section~\ref{sec:assoc} and the main theorem of this paper, in which we prove that the $W$-associahedron is in-your-face (Theorem~\ref{thm:assoc_in_your_face}).  The new idea we bring to this problem concerns how to define the projection---rather than work with type-specific combinatorial models, we define the projection geometrically.  In Section~\ref{sec:combinatorial_models}, we show that our geometric projection recovers a slight modification of the combinatorial projection in~\cite{sleator1988rotation} 

\subsection*{Acknowledgments}

I learned of this problem from T.~Br\"{u}stle at the 2015 Commutative Algebra meets Algebraic Combinatorics conference at Queen's University in Kingston, Ontario, where he presented a proof his summer students had developed in type $D_n$.  T.~Br\"{u}stle has kindly informed me that he and J.-F.~Marceau also have a uniform proof, using different methods~\cite{brustle2015}.  I would like to thank T.~Br\"{u}stle, M.~Guay-Paquet, A.~Pang, and H.~Thomas for useful conversations, C.~Ceballos and V.~Pilaud for their lucid presentation in~\cite{ceballos2014diameter}, and C.~Stump for helpful suggestions on presentation.

\section{The In-Your-Face Property}
\label{sec:inyourface}

Given a (finite) convex polytope $\Po$ with facets $\F=\{f\}_{f \in \F}$ and vertices $\V=\{v\}_{v \in \V}$, recall that its face poset is a finite, graded (by dimension), atomic and coatomic lattice.  That is, an arbitrary face $g$ may be uniquely specified by its set of containing facets \[\F(g)=\{f \in \F : g \subseteq f\},\] or by the set of vertices it contains \[\V(g)=\{v \in \V : v \subseteq g\}.\]  For two vertices $v,v'$, we write $v \to v'$ if they are adjacent on the $1$-skeleton of $\Po$ \textit{or} if $v = v'$.  A \defn{path} from the vertex $v$ to the vertex $v'$ is a sequence \[v=v_1 \to v_2 \to \cdots \to v_k=v'.\]  A \defn{geodesic} from vertices $v$ to $v'$ is a path of minimal length.

\begin{definition}[{~\cite[Definition 5]{ceballos2014diameter}}]
	A polytope $\Po$ is \defn{in-your-face} if for all vertices $v,v' \in \V$, any geodesic from $v$ to $v'$ stays in the minimal face containing both.
\end{definition}

We will not deviate from the original proof strategy of D.~Sleator, R.~Tarjan, and W.~Thurston in~\cite{sleator1988rotation}, which was formalized by C.~Ceballos and V.~Pilaud~\cite{ceballos2014diameter}.  Briefly, for each facet $f$, we define the notion of a ``normalization map'' $\phi_f: \V \to f$, and then show that the existence of such $\phi_f$ implies that $\Po$ is in-your-face. 

\begin{definition}[{~\cite[Lemma 3]{sleator1988rotation},~\cite[Proposition 9]{ceballos2014diameter}}]
\label{def:normalization_map}
	A \defn{normalization map} for the facet $f \in \F$ is a function $\phi_f: \V \to f$ such that:
	\begin{enumerate}
	\item  $\phi_f(v)=v$ for $v \in \V(f)$;
	\item  if $v \to v'$, then $\phi_f(v) \to \phi_f(v')$; and
	\item  if $v \to v'$ with $v \in f$ but $v' \not \in f$, then $\phi_f(v')=v$.
	\end{enumerate}
\end{definition}

\begin{lemma}[{~\cite[Lemma 3]{sleator1988rotation},~\cite[Propositions 8 and 9]{ceballos2014diameter}}]
\label{lem:norm_implies_in_your_face}
	If a normalization map exists for each facet $f \in \F,$ then $\Po$ is in-your-face.
\end{lemma}

\begin{proof}
As in~\cite{sleator1988rotation,ceballos2014diameter}, we first show that if an element $v$ is adjacent (but not on) the facet $f$ containing $v'$, then we may find a geodesic between $v$ and $v'$ whose first step is onto the facet $f$.

To this end, let $v, v', v''$ be three vertices of $\Po$ such that $v''$ and $v$ are adjacent, $v''$ and $v'$ lie on the same face $f$, but $v$ is not on the face $f$.  We show that there is a geodesic between $v$ and $v'$ whose first step is $v \to v''.$

Let \[v=v_0 \to v_1 \to \cdots \to v_k=v'\] be a geodesic between $v$ and $v'$.  Then the sequence (of length one greater than the original sequence) 
\[v=v_0 \to \phi_f(v_0) \to \phi_f(v_1) \to \ldots \to \phi_f(v_k)\]
is actually a path between $v$ and $v'$:
\begin{itemize}
	\item by Property (1) of Definition~\ref{def:normalization_map}, the last element is $v_k$;
	\item by Property (2), the consecutive pairs besides the first are either adjacent or equal; and
	\item by Property (3), the first consecutive pair is adjacent, since we assumed that $v$ is adjacent to a $v'' \in f$, so that $\phi_f(v_0)=v''$.
\end{itemize}

Since the first vertex $v$ was not in the face $f$, but our final vertex $v'$ was in $f$, we passed onto the face $f$ at some step $v_i \to v_{i+1}$ in the original geodesic.  Then $\phi_f(v_i)=\phi_f(v_{i+1})$, and we can remove this duplication to obtain a new geodesic whose first step is the edge $v \to v''.$

We now show that a geodesic between $v$ and $v'$ cannot leave a common facet.  Suppose that $v$ and $v'$ lie on a common facet $f$, and suppose we have a geodesic
\[v=v_0 \to v_1 \to \cdots \to v_k=v'\]
that leaves $f$.  Let $i$ be the minimal index so that the vertex $v_i$ in the geodesic is not on $f$.  By the previous argument, we have that the geodesic from $v_i$ to $v_k$ \[v_i \to v_{i+1} \to \cdots \to v_k\] may be normalized to a new geodesic from $v_i$ to $v_k$ \[v_i=v'_i \to v'_{i+1} \to v'_{i+2} \to \cdots \to v'_k=v_k,\] such that $v'_{i+1}=\phi_f(v_{i+1})$ lies on the face $f$.  By assumption $v_{i-1} \to v_i$ left the face $f$, so by Property (2) we conclude that $\phi_f(v_{i+1})=v_{i-1}=v'_{i+1}.$

Then we can merge the initial segment of the original geodesic from $v=v_0$ to $v_{i-1}$ with the final segment of the new geodesic from $v_{i-1}$ to $v'=v'_k$, so that \[v=v_0 \to v_1 \to \cdots \to v_{i-1} \to v'_{i+2} \to \cdots \to v'_k=v'\] is a path from $v$ to $v'$ of length at least two shorter than the original supposed geodesic.

Finally, since this argument holds for any facet $f$ containing both $v$ and $v'$, it holds true for the minimal face containing both.
\end{proof}

\section{The Permutahedron}
\label{sec:permutahedron}

In this section, we use the framework of the previous section to prove that the $W$-permutahedron is in-your-face.  Although this result can be proven much more quickly and intuitively, we feel justified in this long-winded approach because we will reuse the same ideas in Section~\ref{sec:assoc} for the $W$-associahedron.

Given a \defn{finite Coxeter system} $(W,S)$ with $J \subseteq S$ a subset of the \defn{simple reflections} $S$, the \defn{standard parabolic subgroup} $W_J$ is the Coxeter group generated by $J$, and the \defn{parabolic quotient} is the weak order interval $\{w \in W : sw > w \text{ for all } s \in J\}$.  If $S \setminus J = \{s\}$, we write $W_{\langle s \rangle}$ and $W^{\langle s \rangle}$ for $W_J$ and $W^J$.  
The \defn{length} of $w \in W$ is the length $\ell(w)$ of a shortest expression for $w$ as a product of simple reflections; we write $\Red(w)$ for the set of all such shortest expressions, and call them \defn{reduced words}.  The \defn{weak order} $\Weak(W)$ is the lattice given by relations $u \leq v$ iff $\ell(u)+\ell(u^{-1}v)=\ell(v)$~\cite[Theorem 3.2.1]{bjorner2006combinatorics}.  Recall that we may choose a system of \defn{simple roots} $\Delta(W)$ which determine the \defn{positive roots} $\Phi^+(W).$  We may order $\Phi^+(W)$ by $\alpha<\beta$ iff $\beta-\alpha$ is a nonnegative sum over elements of $\Delta$.  This order has a \defn{highest root}.  We associate to $w \in W$ its \defn{inversion set} $\inv(w)=\left(-w(\Phi^+(W))\right)\cap \Phi^+(W),$ and recall that $W$ has a \defn{longest element} $w_o$ whose inversion set is all of $\Phi^+(W)$.  We will use the notation $u^v:=vuv^{-1}$, and we write $n:=|S|$ and $N:=|\Phi^+(W)|$.

All our examples will be in type $A_3$; in these examples, we label the simple reflections $s_1,s_2,$ and $s_3$, so that $(s_1s_2)^3=(s_2s_3)^3=(s_1s_3)^2=e$.

\begin{proposition}
	\label{prop:proj_unique}
	We recall the following facts with reference but without proof.
\begin{itemize}  	
	\item Every $w \in W$ has a unique factorization $w = w_J \cdot w^J$ with $w_J \in W_J$ and $w^J \in W^J$~\cite[Proposition 2.4.4]{bjorner2006combinatorics}.
	\item The projection from $W \to W_J$ defined by $w \mapsto w_J$ is a lattice homomorphism from $\Weak(W)$ to $\Weak(W_J)$~\cite[Corollary 6.10]{reading2004lattice}.
\end{itemize}
\end{proposition}

We define the \defn{$W$-permutahedron} $\Perm(W)$ to be the convex hull of $W$-orbit of its highest root.  It is well-known that the vertices of $\Perm(W)$ are in bijection with elements $w \in W$, its faces of are parametrized by the cosets $W/W_J$ for $J\subseteq S,$ and its facets are parametrized by the translations of maximal parabolic subgroups $\{w W_{\langle s\rangle} : s \in S, w \in W\}$.

An example is given in Figure~\ref{fig:permpica3}.

\begin{figure}[htbp]
	\begin{center}
		\includegraphics[width=3in]{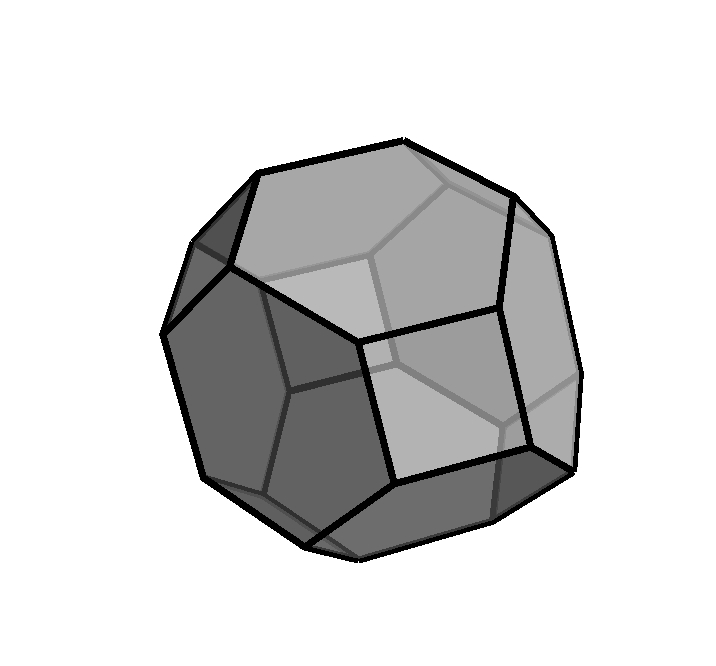}
	\end{center}
	\caption{The permutahedron $\Perm(A_3)$.  The 24 elements of $A_3$ correspond to the 24 vertices.  When correctly oriented, the $1$-skeleton of $\Perm(A_3)$ is isomorphic to $\Weak(A_3)$.}
	\label{fig:permpica3}
\end{figure}

\begin{theorem}
	\label{thm:perm_in_your_face}
The $W$-permutahedron is in-your-face.
\end{theorem}

\begin{proof}
Suppose that $w$ and $w'$ lie in a common facet $f$.  By multiplying $W$ by $w^{-1}$, we orient the permutahedron so that this facet coincides with a maximal parabolic subgroup $W_{\langle s \rangle}$ of $W$.  Since $u \mapsto w^{-1}u$ preserves all edges of the permutahedron, we are reduced to considering only facets of the form $f=W_{\langle s \rangle}$.

In this way, the normalization map naturally suggests itself: we define it as the projection \[\projPerm_{W_{\langle s \rangle}}: W \to W_{\langle s \rangle}\] by \[\projPerm_{W_{\langle s \rangle}}(u) = u_{\langle s \rangle}.\]  By Proposition~\ref{prop:proj_unique}, $\projPerm_{W_{\langle s \rangle}}$ is clearly a map with image $W_{\langle s \rangle}$.  We now check that $\projPerm_{W_{\langle s \rangle}}$ is a normalization map.

\begin{enumerate}
	\item We immediately conclude that $\projPerm_{W_{\langle s \rangle}}(u)=u$ for $u \in W_{\langle s \rangle}$, since $u = u_{\langle s \rangle}$ when $u \in W_{\langle s \rangle}$.
	\item We now check that if $u \to u'$, then $\projPerm_{W_{\langle s \rangle}}(u) \to \projPerm_{W_{\langle s \rangle}}(u')$.  We may write the given edge $u \to u'$ as $u \to ut$, where $ut<u$ and $t\in S$.  Write $u = u_{\langle s \rangle} u^{\langle s \rangle}$; by the exchange lemma, $t$ removes a reflection from either $u_{\langle s \rangle}$ or from $u^{\langle s \rangle}.$  In the former case, the edge is preserved: $u \to ut$ is sent to $u \to u_{\langle s \rangle} t'$, where $t'=t^{u^{\langle s \rangle}}$.  In the latter case, the edge is contracted.

	\item Finally, we ensure that if $u \to u'$ with $u \in W_{\langle s \rangle}$ and $u' \not \in W_{\langle s \rangle}$, then $\projPerm_{W_{\langle s \rangle}}(u')=u$.  We write the given edge $u \to u'$ as $u \to us$, for $u \in W_{\langle s \rangle}$.  Since $s$ does not appear in $u$, $u \cdot s$ is the parabolic decomposition of $us$ with respect to $S \setminus \{s\}$, so that $(us)_{\langle s \rangle} = u$.
\end{enumerate}

Since $\projPerm_{W_{\langle s \rangle}}$ is a normalization map, we conclude the Theorem by Lemma~\ref{lem:norm_implies_in_your_face}.
\end{proof}

An example of this projection is given in Figure~\ref{fig:perma3}.

\begin{figure}[htbp]
	\begin{center}
		\includegraphics[width=.45\textwidth]{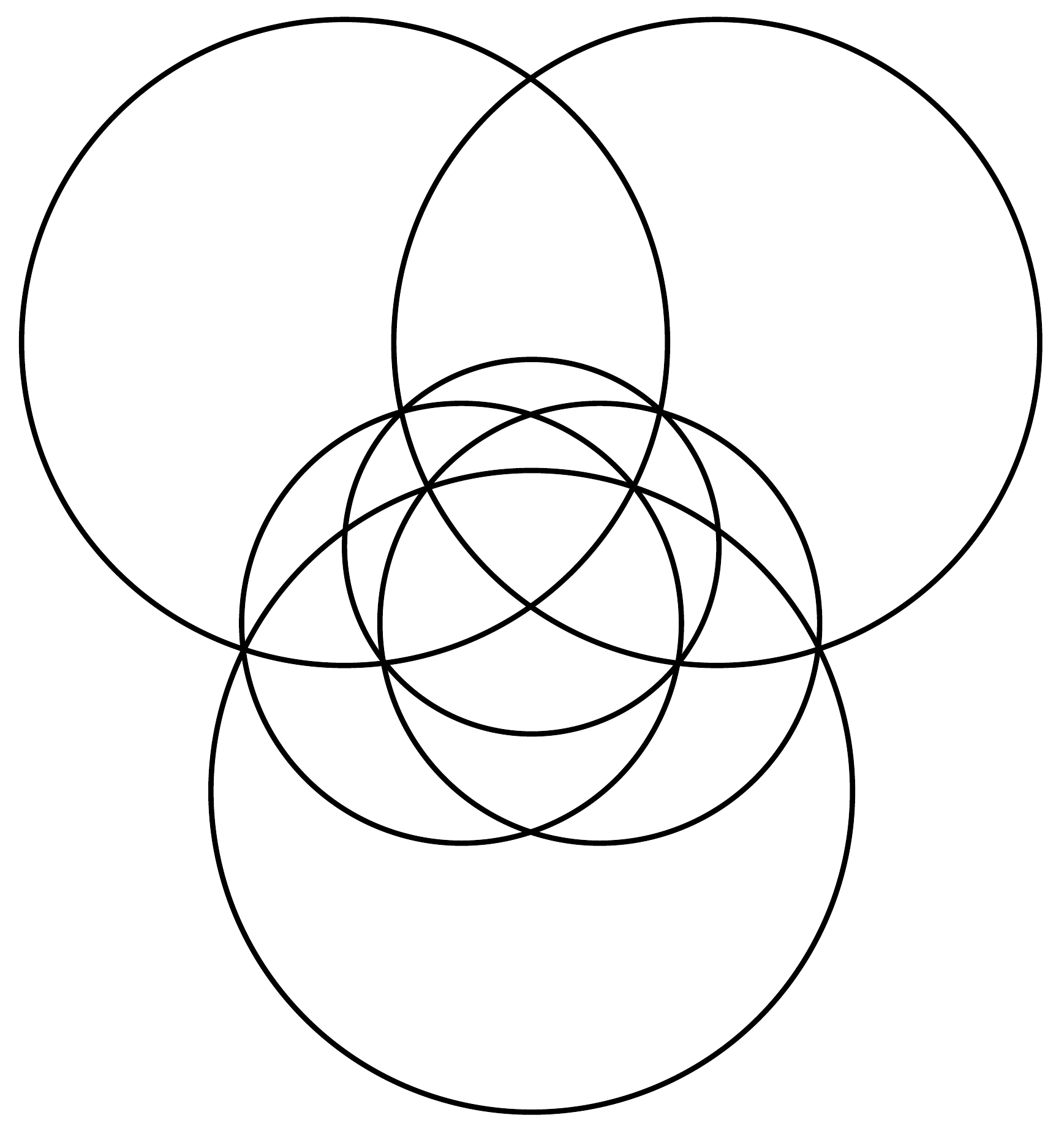} \hspace{1em}
		\includegraphics[width=.45\textwidth]{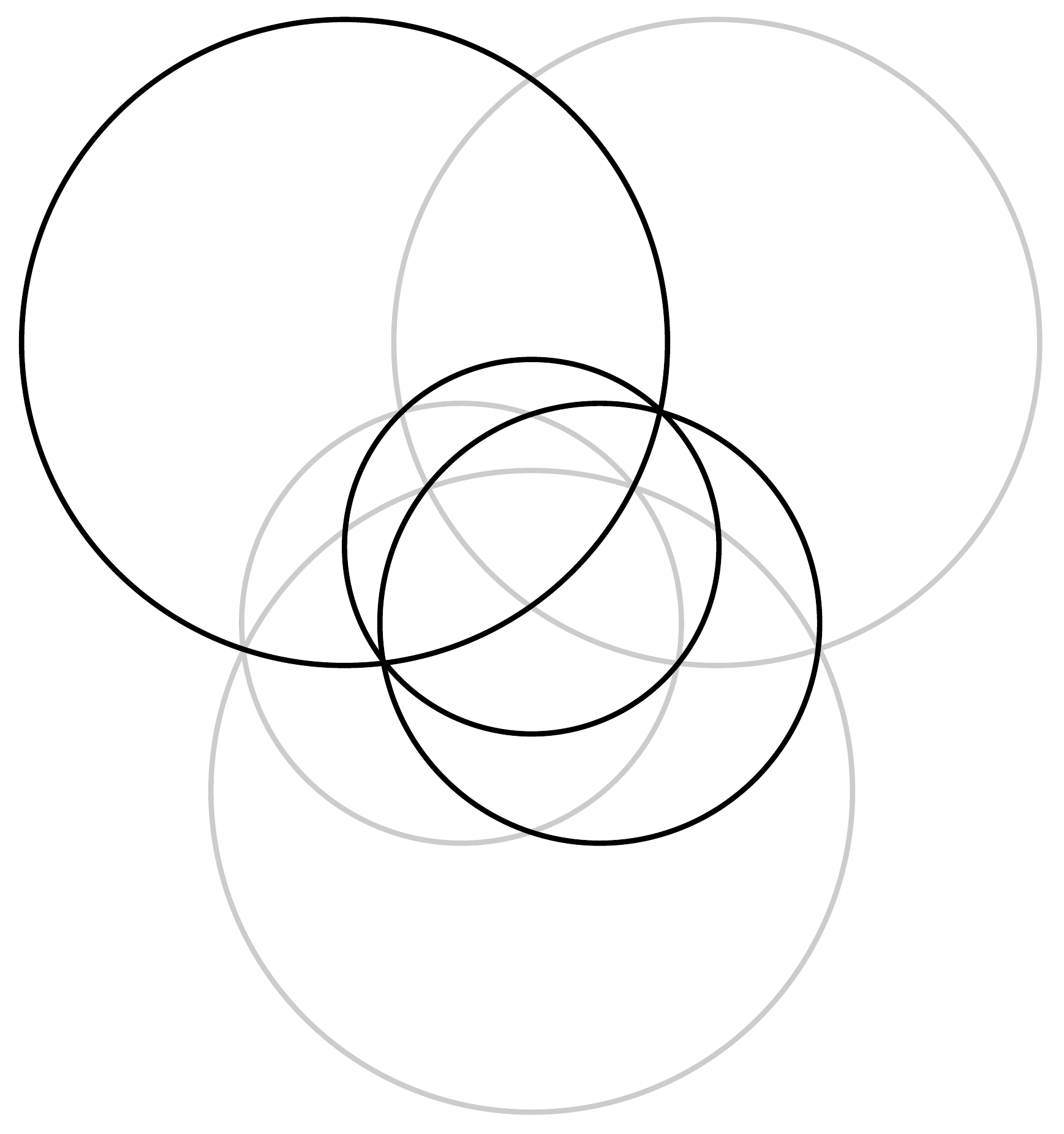}
	\end{center}
	\caption{The picture on the left is the stereographic projection of the intersection of the hyperplane arrangement of type $A_3$ with a sphere (this is the dual of Figure~\ref{fig:permpica3}).  The six hyperplanes correspond to the six circles, and the 24 elements of $A_3$ correspond to the 24 connected regions---the identity element of $A_3$ corresponds to the central triangle, while the longest element corresponds to the unbounded region.   The picture on the right is the restriction of $A_3$ to its maximal parabolic subgroup generated by the reflections $s_2$ and $s_3$.  The projection map sends an element $w$ of $A_3$ to the unique element $w_{\langle s \rangle}$ of the parabolic subgroup $(A_3)_{\langle s_1 \rangle}$ that lies in the same connected region as $w$.}
	\label{fig:perma3}
\end{figure}

\section{The Associahedron}
\label{sec:assoc}

In this section, we will show that the $W$-associahedron is in-your-face.  The proof will follow exactly the same template as our proof for the $W$-permutahedron: we will describe facets, explain how to rotate these facets so that they lie in a parabolic subgroup of $W$, and then define the normalization map by projecting to this parabolic subgroup.

To make this program as smooth as possible, we require two different $(W,c)$-Catalan objects: subword complexes (on which it will be easy to describe the face structure of the $W$-associahedron, its $1$-skeleton, and rotation), and sortable elements (on which it will be easy to understand parabolic subgroups and the normalization map).

\subsection{Subwords and Facets of the Associahedron}

Although there are many different realizations of $W$-associahedra, we only require a description of its faces and of its $1$-skeleton.  We therefore will not attempt to summarize the literature---for more information, the interested reader should consult~\cite{ceballos2011many,hohlweg2011permutahedra}.

A simple description of the faces may be given using the work on subword complexes by J.~P.~Labb\'e, C.~Ceballos, and C.~Stump~\cite{ceballos2014subword}, and by V.~Pilaud and C.~Stump~\cite{pilaud2011brick}.  A (standard) \defn{Coxeter element} $c = s_1 s_2 \cdots s_{n}$ is a product of the simple reflections, each occurring exactly once, in any order. 

Fix a Coxeter element $c$ and choose any reduced word $\mathsf{c}\in \Red(c)$ (this choice is immaterial, since any two such words will agree up to commutations).  The \defn{$c$-sorting word} $\mathsf{w(c)} \in \Red(w)$ of $w$ is the lexicographically first (in position) reduced $S$-subword for $w$ of the word \begin{displaymath}\mathsf{c}^\infty=\Big(s_1 s_2 \cdots s_n \hspace{1ex}|\hspace{1ex} s_1 s_2 \cdots s_n \hspace{1ex}|\hspace{1ex} s_1 s_2 \cdots s_n \hspace{1ex}|\hspace{1ex} \cdots \Big).\end{displaymath}

For example, in type $A_3$ for $c=s_1 s_2 s_3$, $\mathsf{w_o(c)}=s_1 s_2 s_3 | s_1 s_2 \cdot | s_1$.

\begin{definition}[\cite{ceballos2014subword}]
	Fix the word in simple reflections \[\mathsf{Q}=\left(\mathsf{Q}_1,\mathsf{Q}_2,\ldots,\mathsf{Q}_{N+n}\right):=\mathsf{c w_o(c)}.\]  The $(W,c)$-subword complex $\Asoc(W,c)$ consists of all subwords of $\mathsf{Q}$ whose complement contains a reduced word for $w_o$.
\end{definition}

The vertices of the $W$-associahedron are in bijection with those subwords of $\Asoc(W,c)$ with $n$ letters.  Two vertices~$w$ and~$w'$ of $\Asoc(W,c)$ are connected by a \defn{flip} if their corresponding subwords differ in a single position.  The \defn{flip graph} for $\Asoc(W,c)$ is the graph generated by flips, and it is isomorphic to the $1$-skeleton of the $W$-associahedron~\cite{pilaud2011brick}.  An example of the flip graph for $\Asoc(A_3,s_1s_2s_3)$ is given in Figure~\ref{fig:flipgraph}.

\begin{theorem}[\cite{ceballos2014subword,pilaud2011brick}]
In the language of $\Asoc(W,c)$, the faces of the $W$-associahedron are given by those subwords fixing a chosen collection of letters of $\mathsf{Q}$, and its facets are parametrized by single letters of $\mathsf{Q}$.
\end{theorem}

It is not hard to see that the face structure of $\Asoc(W,c)$ does not depend on $c$ (see Section~\ref{sec:orient}), and we may therefore associate $\Asoc(W,c)$ with the \defn{$W$-associahedron}.

\begin{figure}[htbp]
	\begin{center}
		\includegraphics[width=3in]{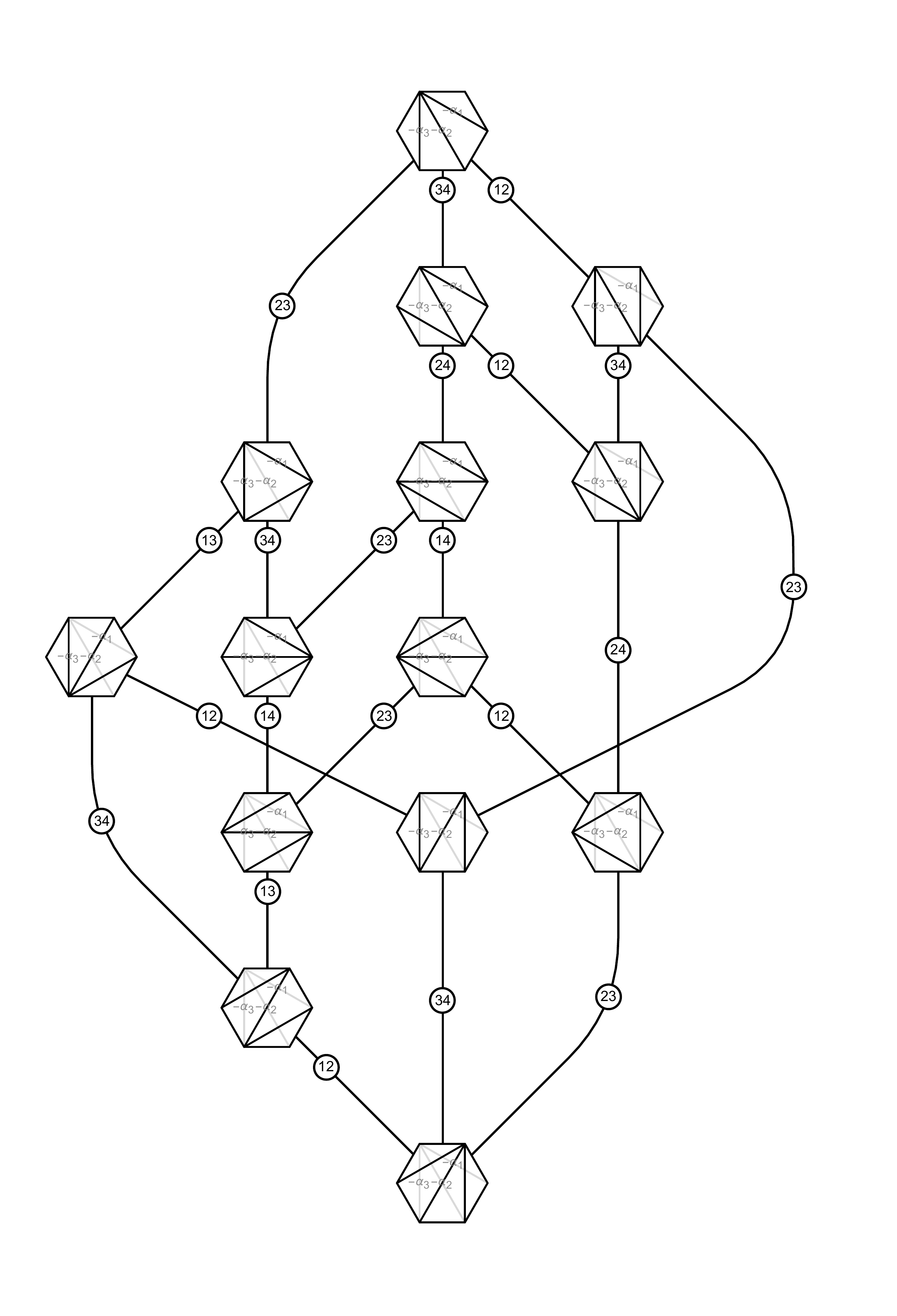}
	\end{center}
	\caption{The flip graph for $\Asoc(A_3,s_1s_2s_3)$, where we have replaced subwords with their corresponding triangulations (a description of this bijection is given in~\cite{ceballos2014subword}; see also Section~\ref{sec:combinatorial_models}).  The letters of $\mathsf{Q}$ correspond to diagonals of the polygon, and moving a letter to flip from one vertex to another may be visualized as flipping the corresponding diagonal.  When correctly oriented as a poset (as it is above), it is isomorphic to the Tamari lattice.  We do not explain the edge labels here, but refer the reader to Figure~\ref{fig:camba3}.}
		\label{fig:flipgraph}
\end{figure}

\subsection{Sortable Elements and Parabolic Subgroups}

In order to describe the normalization map, we require the notion of a parabolic subgroup on $(W,c)$-Catalan objects.  This is most easily introduced using N.~Reading's $c$-sortable elements $\Sort(W,c) \subseteq W$, which---as a subset of elements of $W$---inherit the parabolic structure of $W$.

\begin{definition}[{\cite{reading2007sortable}}]
\label{def:sort_sortable}
   An element $w \in W$ is \defn{$c$-sortable} if its $c$-sorting word defines a decreasing sequence of subsets of positions in $\mathbf{c}$.  We denote the set of $c$-sortable elements by $\Sort(W,c)$.
\end{definition}

Let $c'$ be the \defn{restriction} of $c$ to $W_J$, obtained by removing all simple reflections not in $J$ from a reduced word $\mathsf{c}$.  It is immediate from this definition that if $w$ is $c$-sortable, then $w_J$ is $c'$-sortable, and so we obtain the desired parabolic structure.  

We now relate $\Sort(W,c)$ and $\Asoc(W,c)$ to characterize the parabolic structure on $\Asoc(W,c)$.  There is a natural bijection between the vertices of $\Asoc(W,c)$ and $\Sort(W,c)$, which may be described using equivalent recursive structures on each set of objects, as in~\cite[Section 6.4.2]{pilaud2011brick} or~\cite{reading2007clusters}.  In particular, an element $w \in \Sort(W,c)$ is in $W_J$ iff its corresponding subword (with $n$ letters) does not use the simple reflections in $J$ from the inital copy of $\mathsf{c}$ in $\mathsf{Q}=\mathsf{c}\mathsf{w_o(c)}$.  We conclude the following Proposition.

\begin{proposition}[\cite{ceballos2014subword,pilaud2011brick}]
\label{prop:parabolic_on_asoc}
	Let $s$ be initial in $c$ and write $\mathsf{Q}=\mathsf{c w_o(c)}$ so that its first letter $\mathsf{Q}_1$ is $s$.  The facet $\{\mathsf{Q}_1\}$ of $\Asoc(W,c)$ contains the vertices recursively in bijection with those $c$-sortable elements in $W_{\langle s \rangle}$.
\end{proposition}

In order to check that projection to the parabolic subgroup is a normalization map, we require reasonably explicit descriptions of the edges of the $1$-skeleton of the $W$-associahedron---and hence the edges of the flip graph of $\Asoc(W,c)$---as $c$-sortable elements, and slightly more explicit descriptions of the edges incident to the parabolic subgroup $W_{\langle s \rangle}$.

Recall that $W$ has the projection $\pi_\downarrow^c: W \to \Sort(W,c)$, where \[\pi_\downarrow^c(w)=\begin{cases} s \pi_\downarrow^{scs}(sw) & \text{if } w \geq s \\ \pi_\downarrow^{sc}(w_{\langle s \rangle}) & \text{if } w \not \geq s,\end{cases}\] which by ~\cite[Corollary 6.2]{reading2011sortable} may be characterized as the unique maximal $c$-sortable element below $w$ in the weak order.  This projection respects the parabolic structure of $W$, in that it satisfies \[\pi_\downarrow^c(w)_J = \pi_\downarrow^{c'}(w_J),\] where $c'$ is the restriction of $c$ to $W_J$~\cite[Proposition 6.13]{reading2011sortable}.

\begin{proposition}
\label{prop:edges_on_sort}
	\begin{enumerate}
		\item The underlying graph of the restriction of $\Weak(W)$ to $\Sort(W,c)$ is the $1$-skeleton of the $W$-associahedron.  Every edge may be expressed uniquely as $w \to \pi_\downarrow^c(ws)$ for some $w \in W$ and some $s \in S$ such that $ws <w$.
		\item Let $s$ be initial in $c$ and let $w \in \Sort(W_{\langle s \rangle},sc)$.  Then $v=s \vee w$ is the unique edge from $w$ to an element not in $W_{\langle s \rangle}$.
	\end{enumerate}
\end{proposition}

\begin{proof}
	The first statement follows~\cite[Corollary 8.1]{reading2006cambrian}.
	
	For the second statement, we first note that since $v=s \vee w$ is the join of two $c$-sortable elements, it is itself $c$-sortable.  Since a Cambrian lattice in rank $n$ is $n$-regular, there is exactly one edge from $w$ to an element outside of $W_{\langle s \rangle}$.  We check that $w \to v$ is indeed this edge, which would follow if $\pi_\downarrow^c(v t) = w$ for some simple reflection $t$.  Since $s$ is a cover reflection of $v$ by~\cite[Lemma 2.8]{reading2007sortable}, there is a simple reflection $t$ so that $sv=vt$.  Then $s \not \in \inv(vt)$, so that \[\pi_\downarrow^c(vt)=\pi_\downarrow^{sc}((vt)_{\langle s \rangle})=\pi_\downarrow^{sc}(w)=w,\]
	since \[\inv((vt)_{\langle s \rangle}) = \inv(vt)\cap \Phi^+(W_{\langle s \rangle})=\inv(v) \cap \Phi^+(W_{\langle s \rangle})=\inv(v_{\langle s \rangle})=\inv(w).\]
\end{proof}

\subsection{Orienting the Associahedron}
\label{sec:orient}

To orient the $W$-associahedron---as we did with the $W$-permutahedron by multiplication---we require the notion of \defn{Cambrian rotation}.  Write $\overline{s}=s^{w_o} \in S$ and shift a subword in $\Asoc(W,c)$ one position to the left to obtain a subword of $\mathsf{Q'}=s_{\pi_2} \cdots s_{\pi_{n}} \mathsf{w_o(c)} \overline{s}$ (the leftmost letter $s$ is sent to the rightmost letter $\overline{s}$).   If we let $c'=s^{-1} cs$ with corresponding reduced word $\mathsf{c'}$ then, up to commutations, $s_{\pi_2} \cdots s_{\pi_{n}} \mathsf{w_o(c)} \overline{s}=\mathsf{c' w_o(c')},$ so that the resulting subword corresponds to a subword in $\Asoc(W,c')$~\cite{ceballos2014subword}.  This defines a map \[\Camb_s: \Asoc(W,c) \to \Asoc(W,s^{-1}cs),\] and it is easy to see that this map does not change vertex adjacencies or the face structure.  Using the characterization of facets from Proposition~\ref{prop:parabolic_on_asoc}, the following Lemma follows immediately from the definition of Cambrian rotation.

\begin{lemma}
\label{lem:camb_facet_to_para}
	Any facet $f$ of $\Asoc(W,c)$ may be sent by Cambrian rotation to a facet $f'$ of $\Asoc(W,c')$ for some $c'$, such that the vertices of $f'$ correspond to $sc'$-sortable elements in $W_{\langle s \rangle},$ where $s$ is initial in $c'$.
\end{lemma}


\subsection{The Main Theorem}
\label{sec:main_theorem}

We are now ready to prove the main theorem of this note.

\begin{theorem}
\label{thm:assoc_in_your_face}
The $W$-associahedron is in-your-face.
\end{theorem}

\begin{proof}
	Suppose that we have two vertices $w$ and $w'$ that lie in a common facet $f$ of the $W$-associahedron.  Thinking of $w$ and $w'$ as vertices of the subword complex, by Lemma~\ref{lem:camb_facet_to_para} we can orient the associahedron so that $f$ coincides with the set of sortable elements $\Sort(W_{\langle s \rangle},sc)$ in a maximal standard parabolic.  Since Cambrian rotation preserves all edges of the $W$-associahedron, we are reduced to considering only facets of the form $f=\Sort(W_{\langle s \rangle},sc)$.
	
	The normalization map again naturally suggests itself: we define \[\projCamb_{W_{\langle s \rangle}}: \Sort(W,c) \to \Sort(W_{\langle s \rangle},sc)\] by \[\projCamb_{W_{\langle s \rangle}}(u) = u_{\langle s \rangle}.\]  By the discussion directly after Definition~\ref{def:sort_sortable}, $\projCamb_{W_{\langle s \rangle}}$ is a map with image to $\Sort(W_{\langle s \rangle},sc)$.  We now check that $\projCamb_{W_{\langle s \rangle}}$ is a normalization map.
	
	\begin{enumerate}
		\item We immediately conclude that $\projCamb_{W_{\langle s \rangle}}(u)=u$ for $u \in W_{\langle s \rangle}$, since $u = u_{\langle s \rangle}$ when $u \in W_{\langle s \rangle}$.
		
		\item We now check that if $u \to u'$, then $\projCamb_{W_{\langle s \rangle}}(u) \to \projCamb_{W_{\langle s \rangle}}(u')$.  By Proposition~\ref{prop:edges_on_sort}, we may write such an edge as $u \to \pi_\downarrow^c(ut)$, where $ut<u$ and $t\in S$.  Now 
		\[\pi_\downarrow^c(u t)_{\langle s \rangle} = \pi_\downarrow^{sc}((u t)_{\langle s \rangle}),\] by Proposition 6.13 in~\cite{reading2011sortable}.  Using the same argument as in Theorem~\ref{thm:perm_in_your_face}, we write $u = u_{\langle s \rangle} u^{\langle s \rangle}$ so that $t$ removes a reflection from either $u_{\langle s \rangle}$ or from $u^{\langle s \rangle}.$  In the former case, since $u_{\langle s \rangle}$ is $c$-sortable, the edge is preserved: $u \to \pi_\downarrow^c(ut)$ is sent to $u_{\langle s \rangle} \to \pi_\downarrow^c(u_{\langle s \rangle} t')$, where $t'=t^{u^{\langle s \rangle}}$.  In the latter case, the edge is contracted.

		\item Finally, we ensure that if $u \to u'$ with $u \in W_{\langle s \rangle}$ and $u' \not \in W_{\langle s \rangle}$, then $\projCamb_{W_{\langle s \rangle}}(u')=u$.
		By Proposition~\ref{prop:edges_on_sort}, we may write such an edge as $u \to u'$, where $u'=s \vee u$.  Now $u' = s \vee u'_{\langle s \rangle}$~\cite[Lemma 2.7]{reading2007sortable}, so that because
		\begin{itemize}
			\item $u'_{\langle s \rangle}$ is $sc$-sortable in $W_{\langle s \rangle}$,
			\item $c$-sortable elements are uniquely characterized by their cover reflections~\cite[Proposition 2.5]{reading2007sortable}, and
			\item $\cov(u')=\cov(s \vee u)=\cov(u) \cup \{s\}$~\cite[Lemma 2.8]{reading2007sortable},
		\end{itemize}
		we conclude that \[u=u'_{\langle s \rangle}=\projCamb_{W_{\langle s \rangle}}(u').\]
	\end{enumerate}
	
	Since $\projCamb_{f}$ is a normalization map, we conclude the Theorem by Lemma~\ref{lem:norm_implies_in_your_face}.
	\end{proof}

An example of this projection is given in Figures~\ref{fig:camba3} and~\ref{fig:camba3proj}.

\begin{figure}[htbp]
	\begin{center}
	\includegraphics[width=\textwidth]{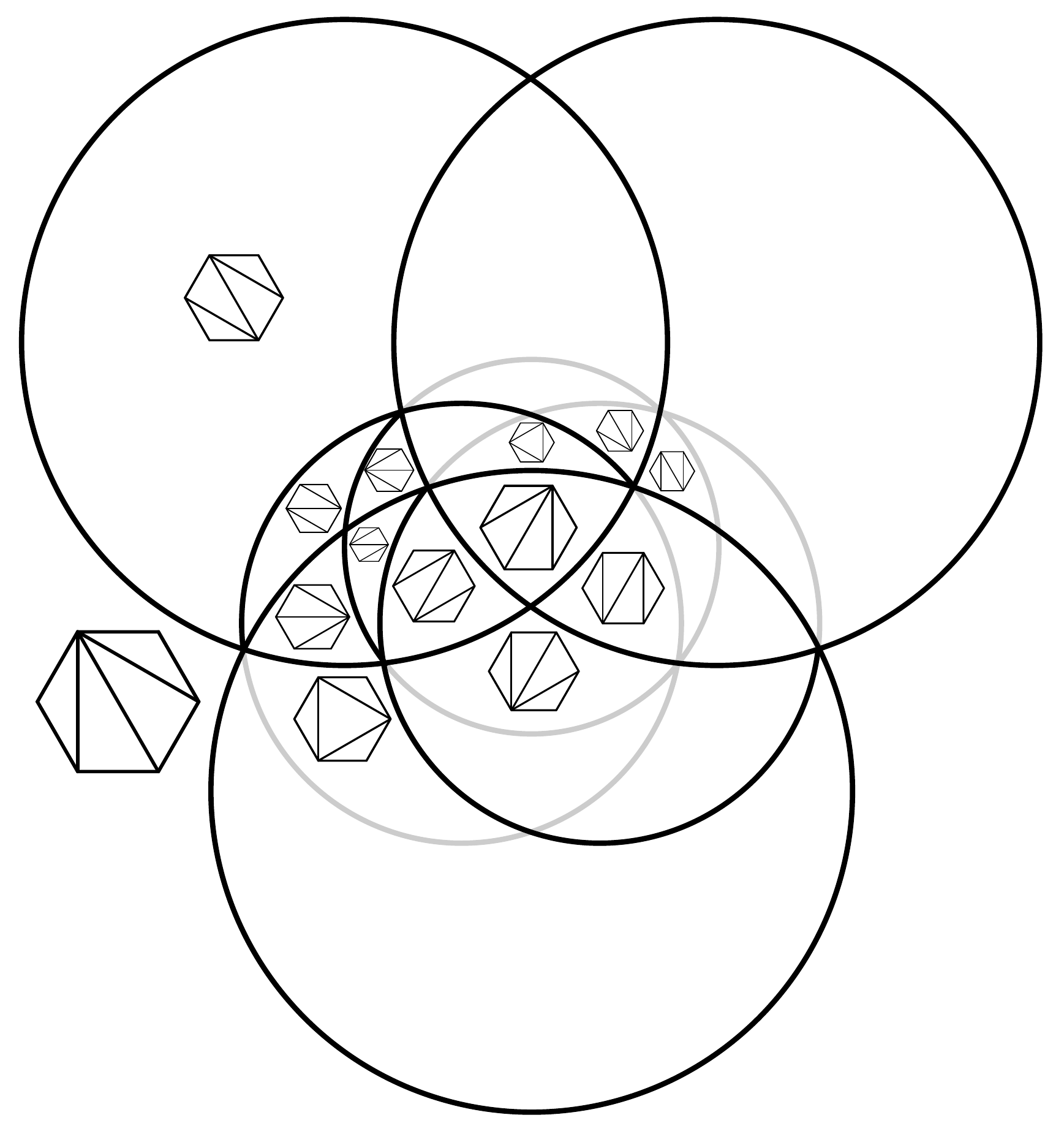}
	\end{center}
	\caption{The stereographic projection of the intersection of the hyperplane arrangement of type $A_3$ with a sphere, where the arcs in black are those edges of the permutahedron preserved under the Cambrian congruence for $c=s_1s_2s_3$.  The connected regions (using only the arcs in black) are indexed by vertices of the subword complex (which we draw as triangulations, as in Figure~\ref{fig:flipgraph}).  The minimal connected regions (when using both gray and black arcs) in the connected regions (using only black arcs) are the $c$-sortable elements.}
	\label{fig:camba3}
\end{figure}
	
\begin{figure}[htbp]
	\begin{center}
	\includegraphics[width=\textwidth]{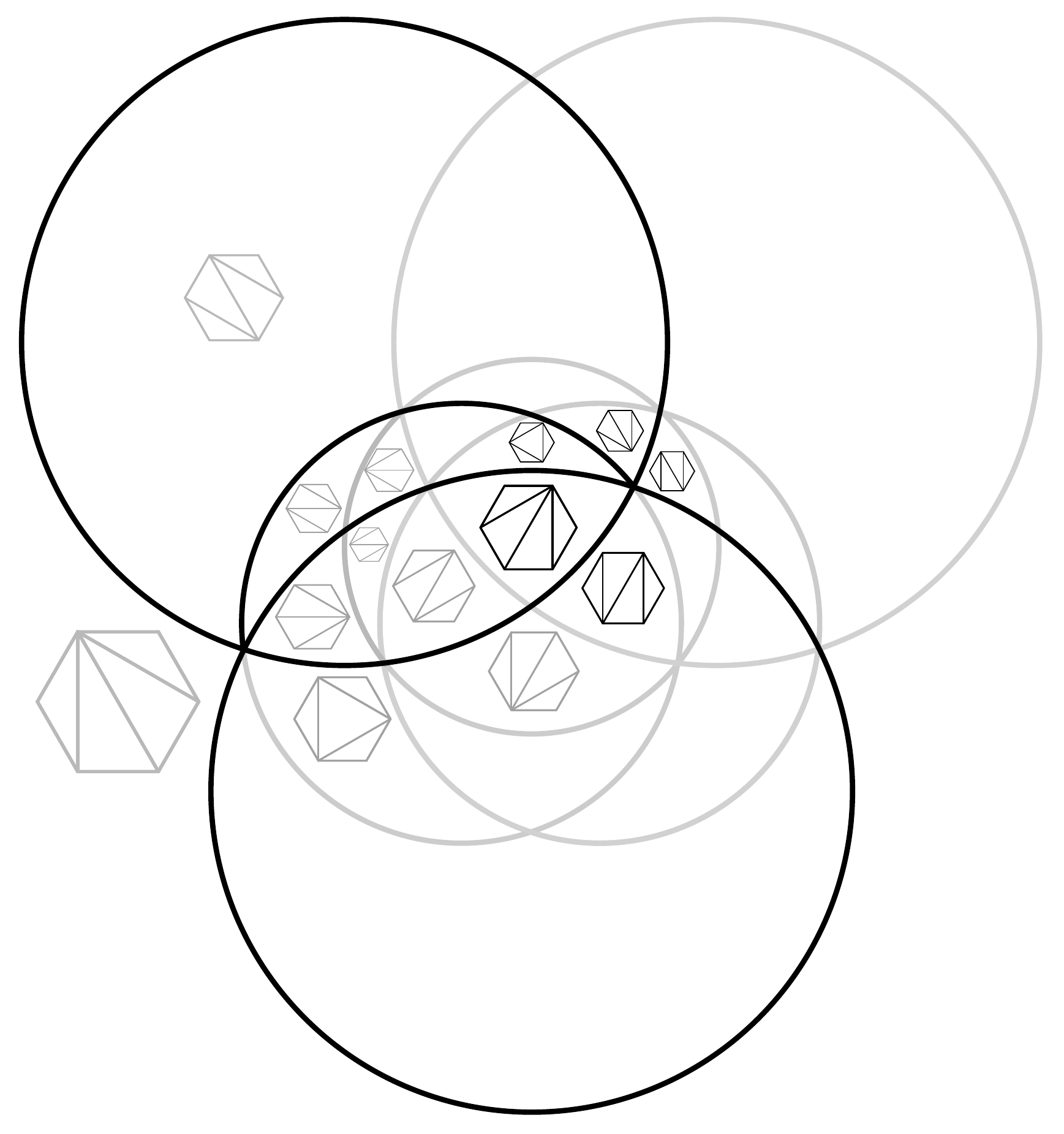}
	\end{center}
	\caption{The restriction of Figure~\ref{fig:camba3} to the maximal parabolic subgroup generated by $s_2$ and $s_3$.  The projection map sends an element $w$ of $\Sort(A_3,c)$ to the unique element $w_{\langle s \rangle}$ of $\Sort((A_3)_{\langle s_1 \rangle},s_2s_3)$ that lies in the same connected region as $w$. It is geometrically clear that this projection either preserves or contracts edges.}
	\label{fig:camba3proj}
\end{figure}

\section{Combinatorial Models}
\label{sec:combinatorial_models}

In this section, we review the combinatorics of type $A$, and prove that our projection recovers a slight modification of the projection defined in~\cite{sleator1988rotation}. 

\subsection{Triangulations}
\label{sec:triangulations}
We first recall the combinatorial model for triangulations in type $A_{n-1}$ for a standard Coxeter element $c$~\cite{reading2006cambrian,ceballos2014subword}.  The element $c$ is an $n$-cycle of the form \[c:=(1, d_1, d_2, \cdots, d_l, n, u_k, u_{k-1}, \cdots, u_{1})\] so that the \defn{lower numbers} $d_i$ satisfy $d_1<d_2<\cdots<d_l$ and the \defn{upper numbers} $u_i$ satisfy $u_1<u_2<\cdots<u_{k}$.  

We use the $n$-cycle above to describe an unconventional method for labeling an $(n+2)$-gon.  We will first describe the vertices of this $(n+2)$-gon.  We fix a circle of radius $(n+1)$ with left-most point at $(0,0)$.  We will refer to the point above the $x$-axis with $x$-coordinate equal to $i$ as the \defn{upper vertex} $i$, and similarly call the corresponding point below the $x$-axis the \defn{lower vertex} $i$.  We mark the vertices $(0,0)$ and $(n+1,0)$, as well as all upper vertices $d_i$ and all lower vertices $u_i$.  
We may freely choose to mark either the upper or the lower vertex of the circle for the $x$-coordinates $1$ and $n$, so that there is exactly one marked vertex on the circle for each $x$-coordinate from $0$ to $n+1$.

We now recall N.~Reading's elegant bijection between $c$-sortable elements and triangulations of this $(n+2)$-gon~\cite{reading2006cambrian}.  Fix a $c$-sortable element $w$, let $\lambda_0$ be the path from $(0,0)$ to $(n+1,0)$ that passes along the lower part of the circle, and let $w=w_1w_2\cdots w_n$ be the one-line notation for $w$.  We now read the one-line notation for $w$ from left to right.  At the $k$-th step, we have a previously-constructed path $P_{k-1}$, and we have reached the $k$th letter $w_k$ of $w$'s one-line notation.  If $w_k$ is a lower number, then $P_{k}$ is equal to $P_{k-1}$ without the lower vertex $w_k$; otherwise, if $w_k$ is an upper number, then $P_k$ is equal to $P_{k-1}$ with the upper vertex $w_k$ added.  Then the union of the paths $P_k$ for $1\leq k \leq n$ is the drawing of a triangulation of the unconventionally-labeled $(n+2)$-gon.  We shall denote this triangulation $\Tri_c(w).$  An example is given in Figure~\ref{fig:c_sort_to_triangulation}.

\begin{figure}[htbp]
\[\begin{array}{ccccc}
\includegraphics[width=.3\textwidth]{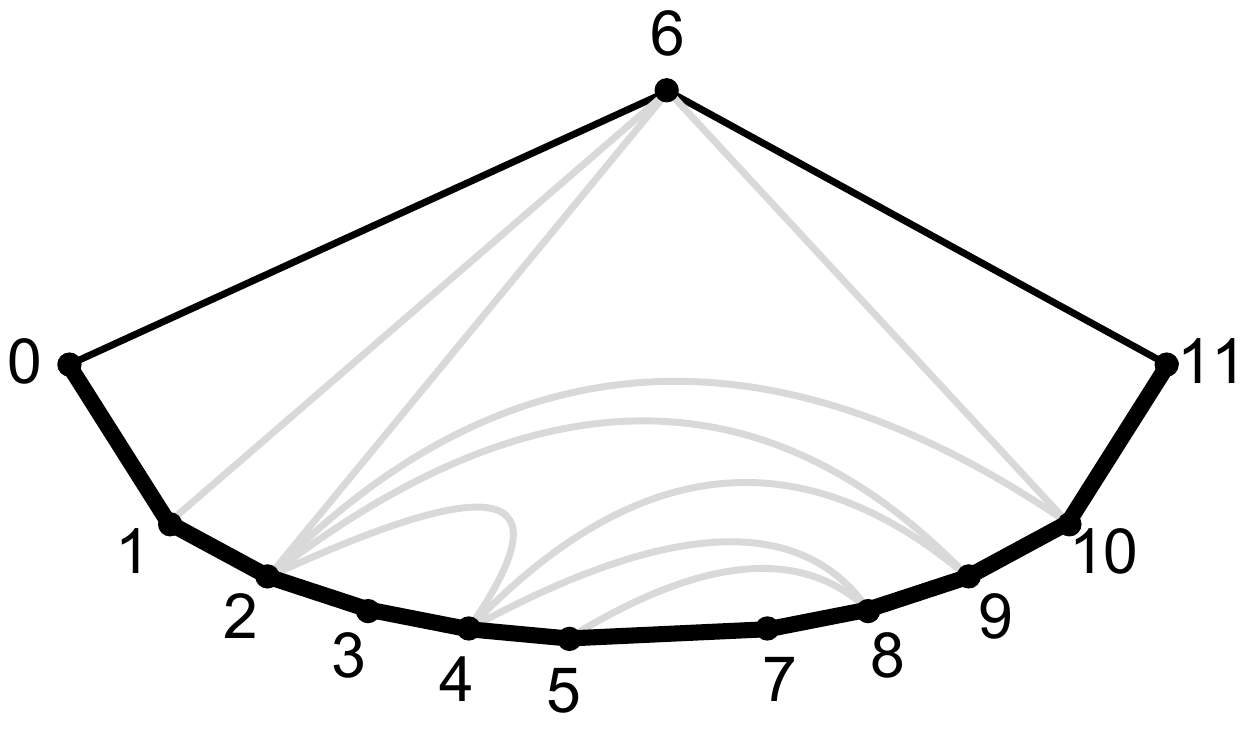} &\includegraphics[width=.3\textwidth]{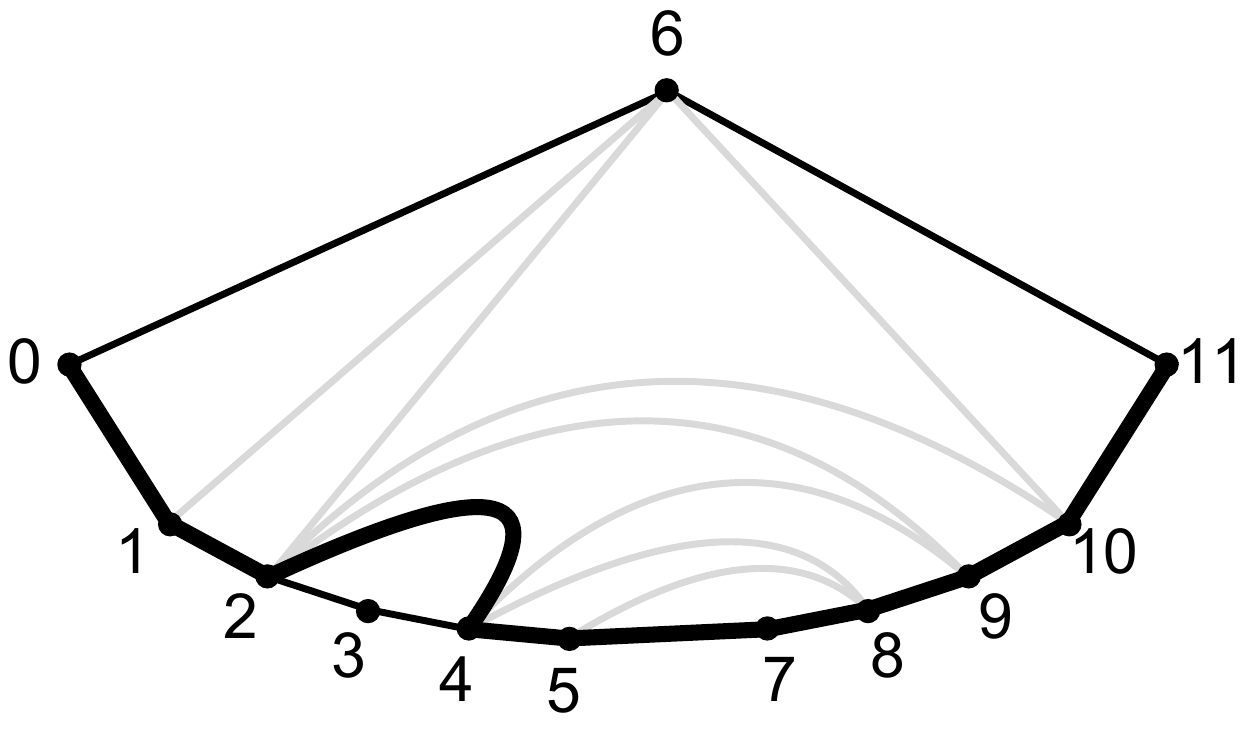}&\includegraphics[width=.3\textwidth]{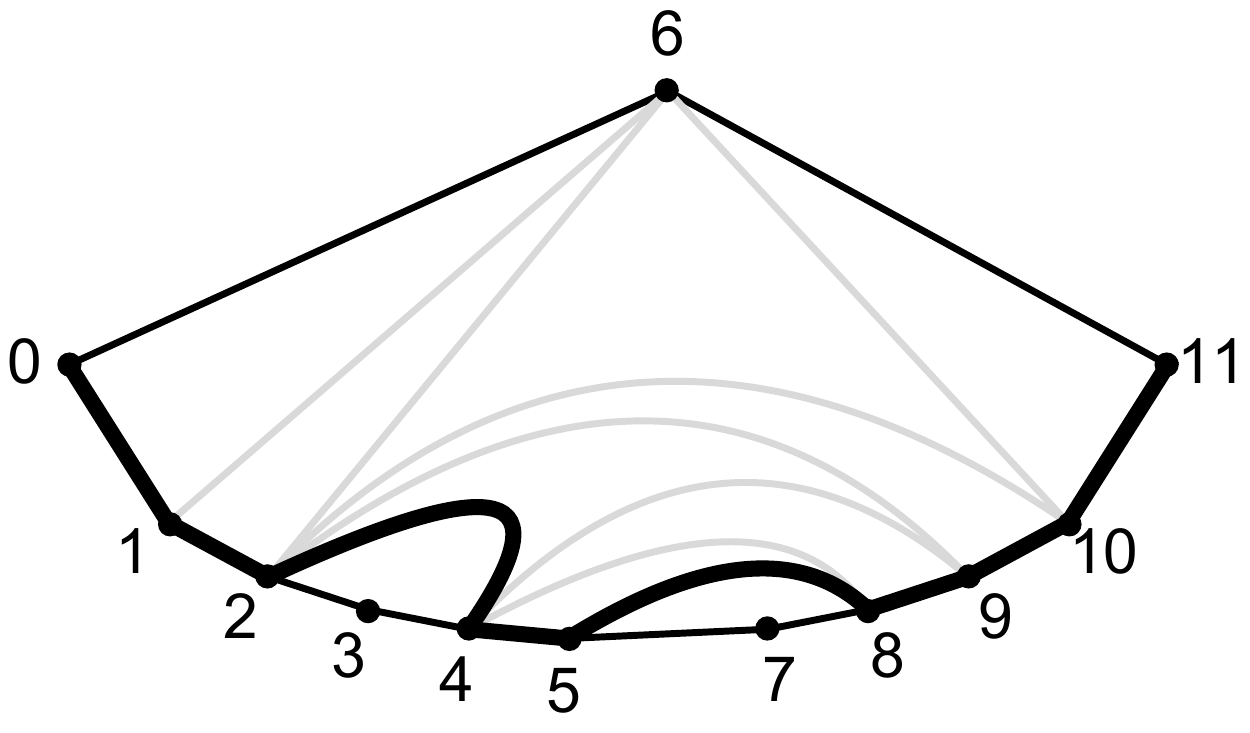}\\
\includegraphics[width=.3\textwidth]{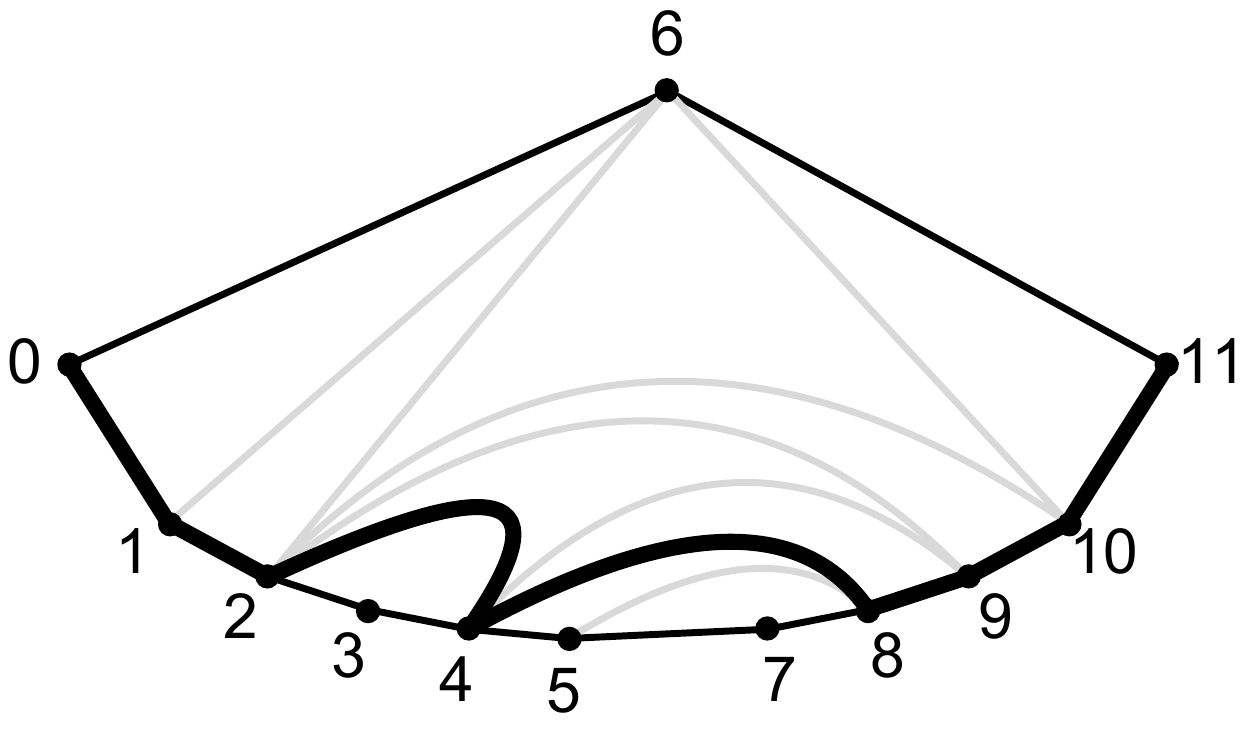}&\includegraphics[width=.3\textwidth]{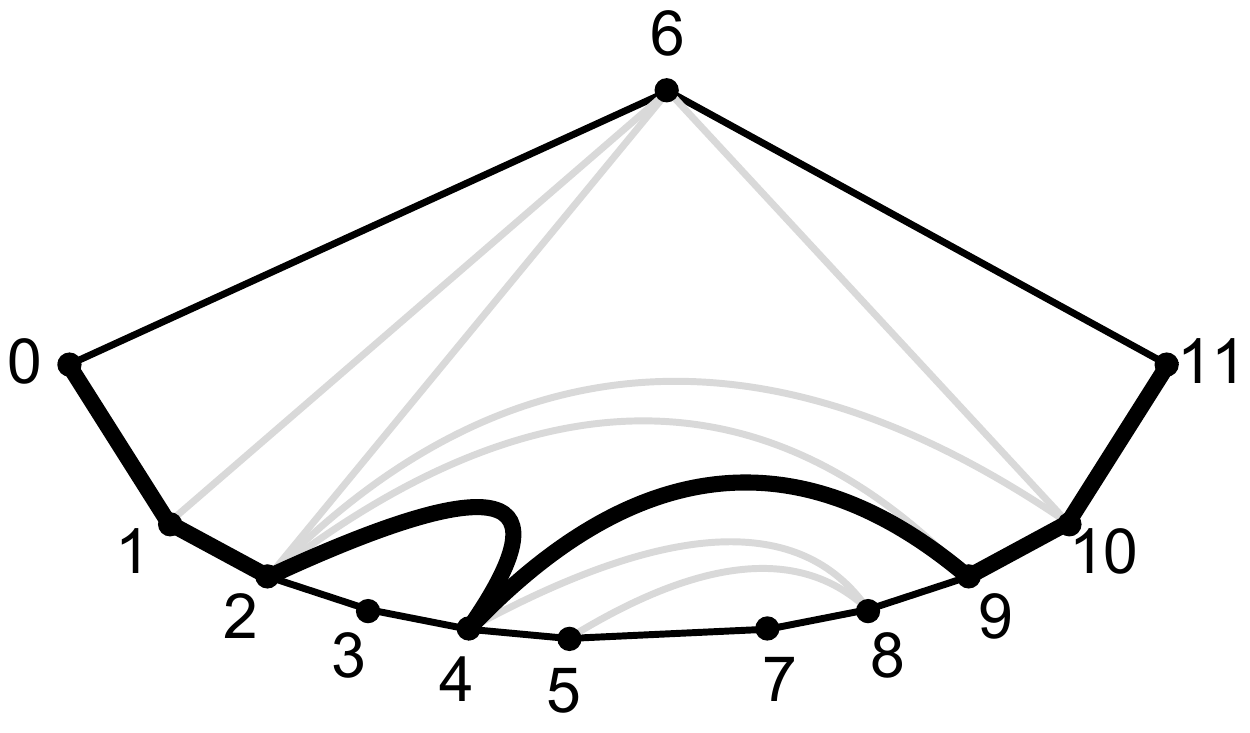} &\includegraphics[width=.3\textwidth]{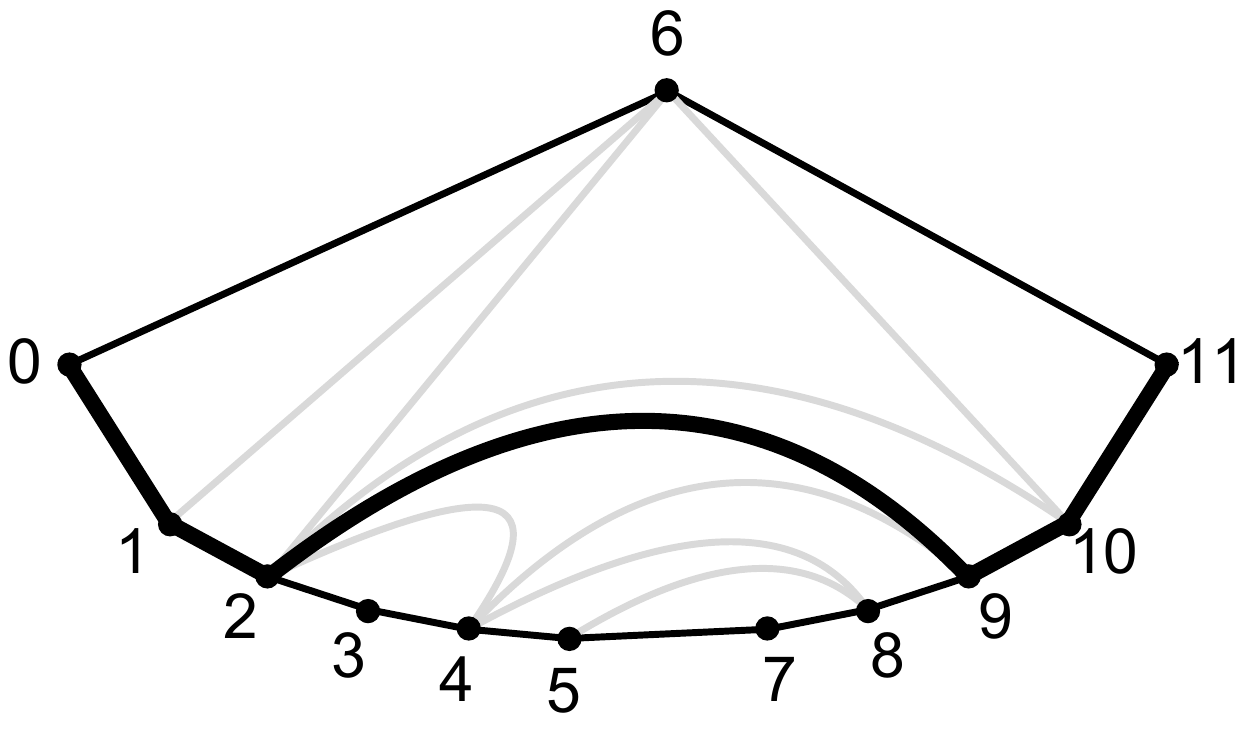}\\
\includegraphics[width=.3\textwidth]{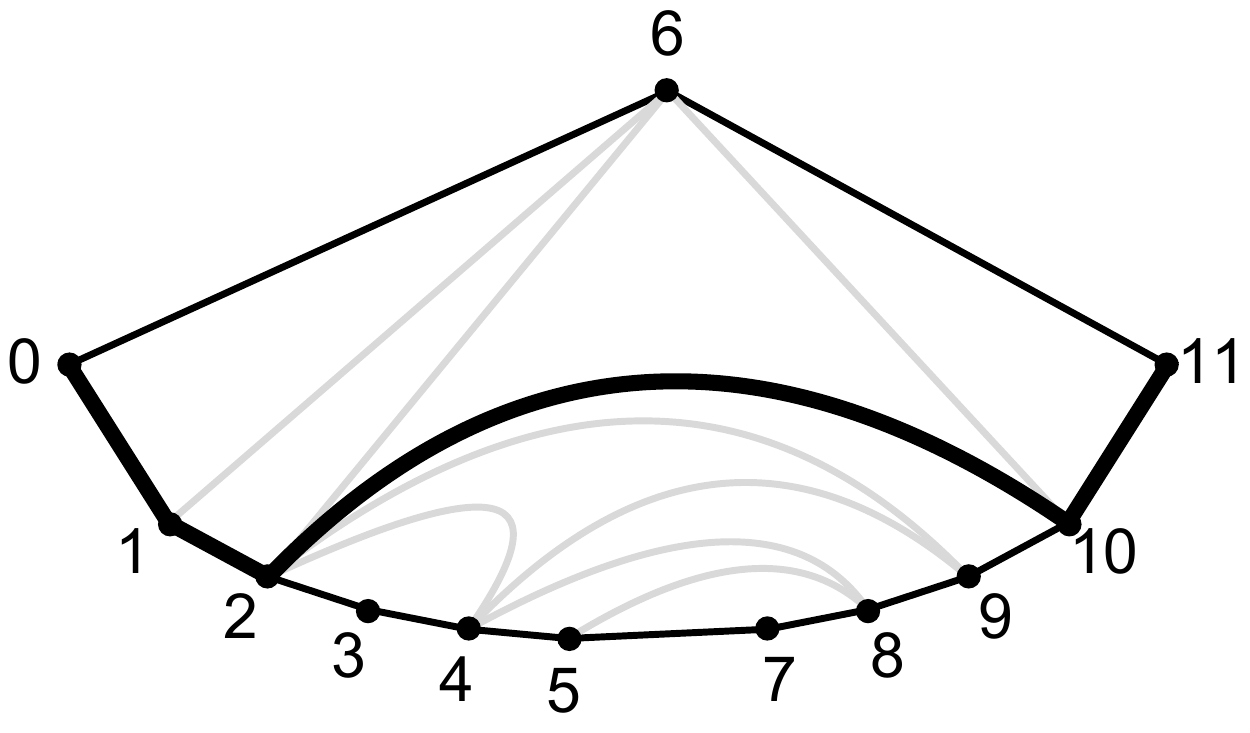} 
&\includegraphics[width=.3\textwidth]{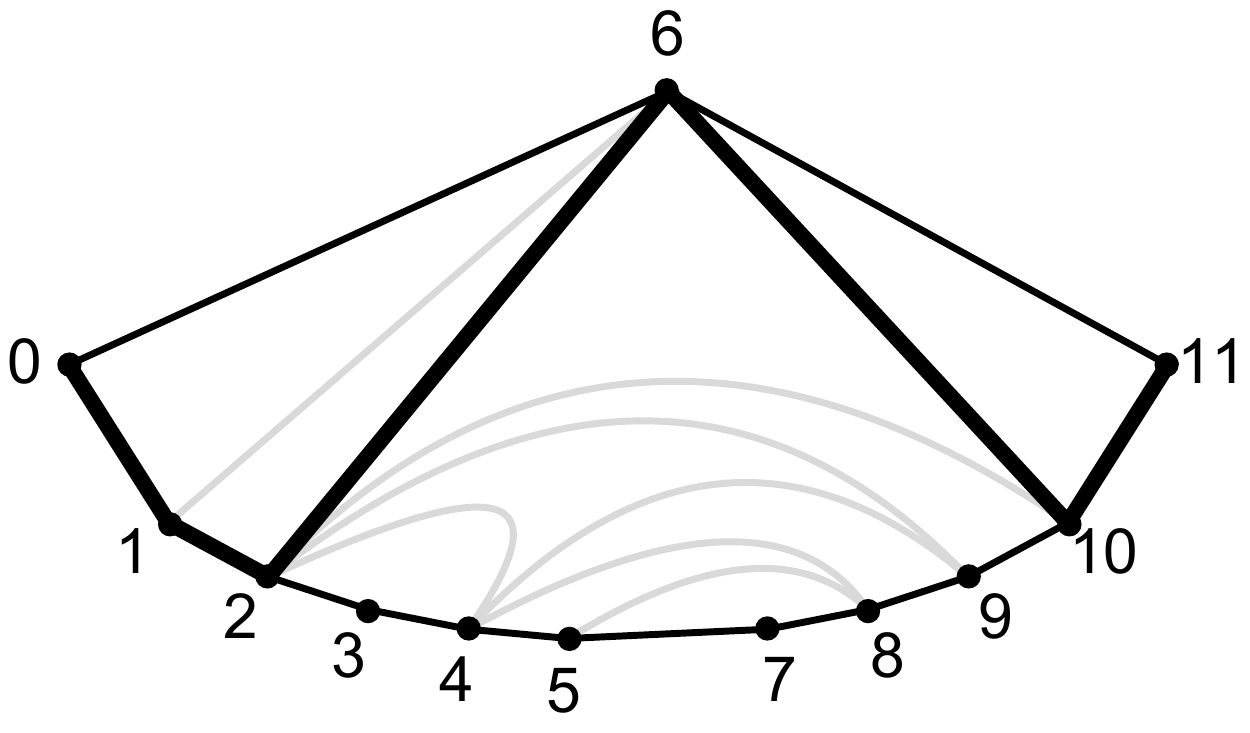} 
&\includegraphics[width=.3\textwidth]{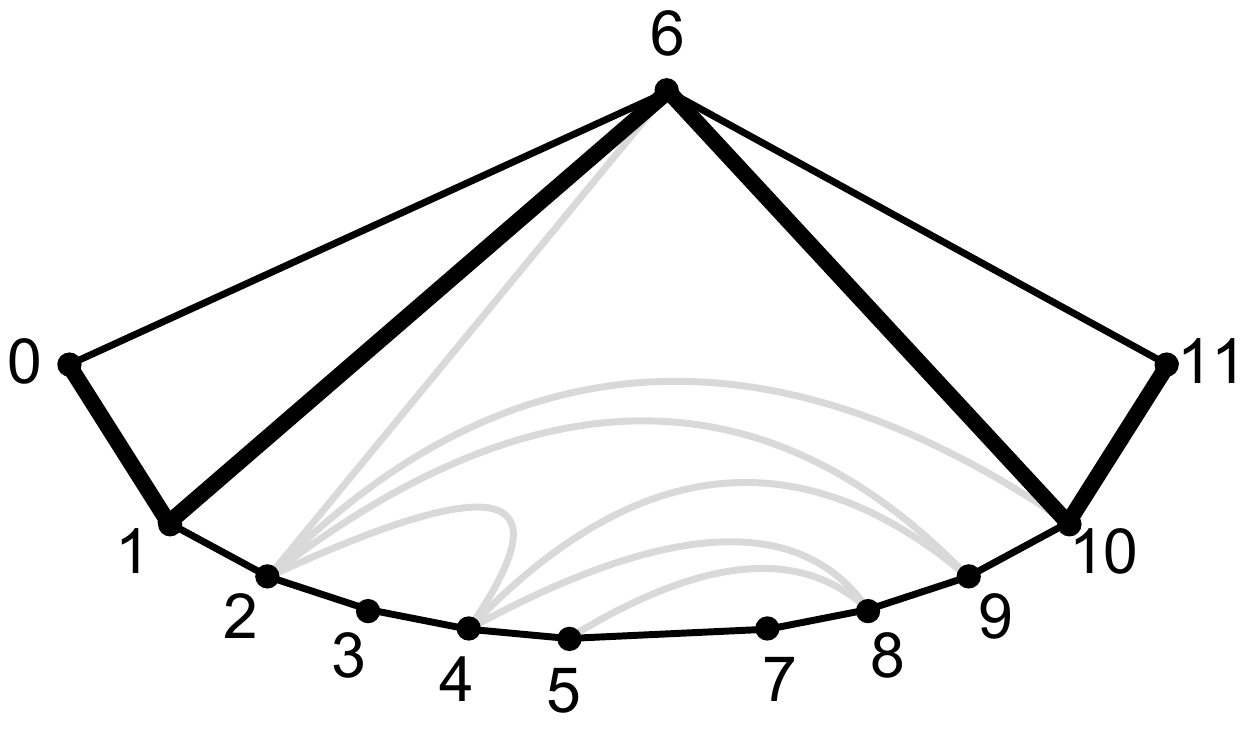}\\
\includegraphics[width=.3\textwidth]{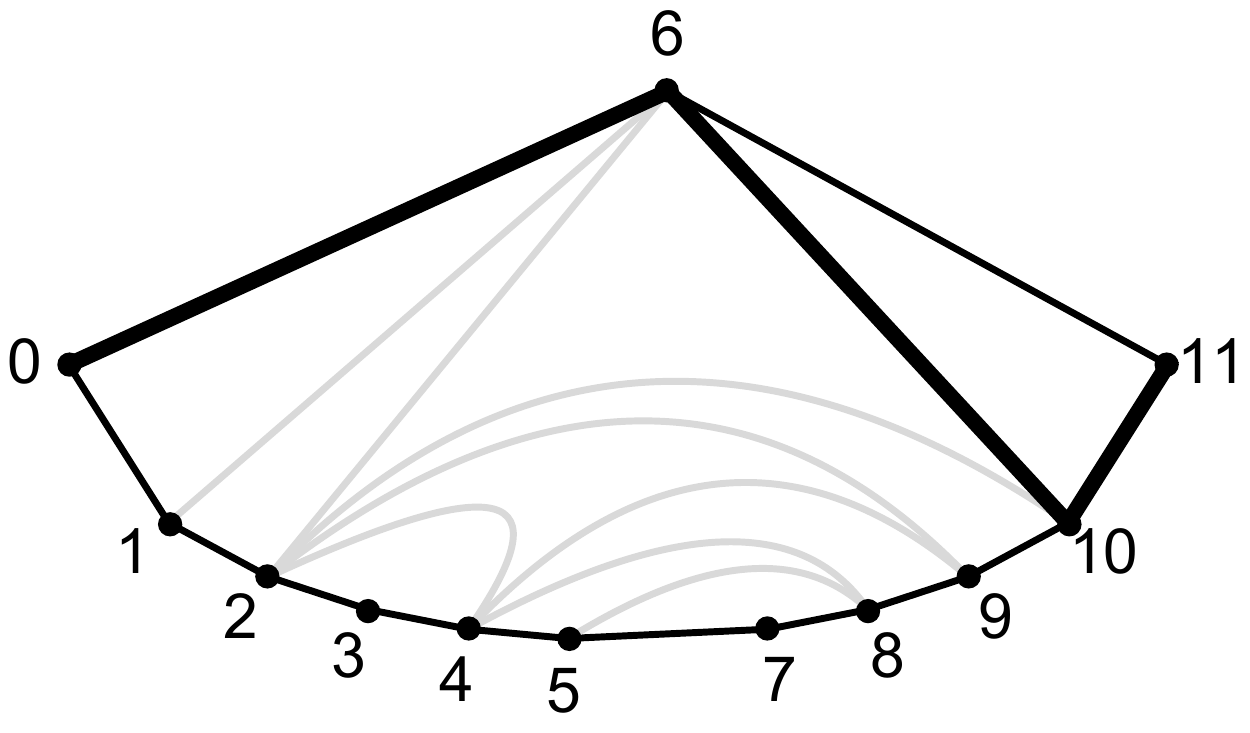}
&\includegraphics[width=.3\textwidth]{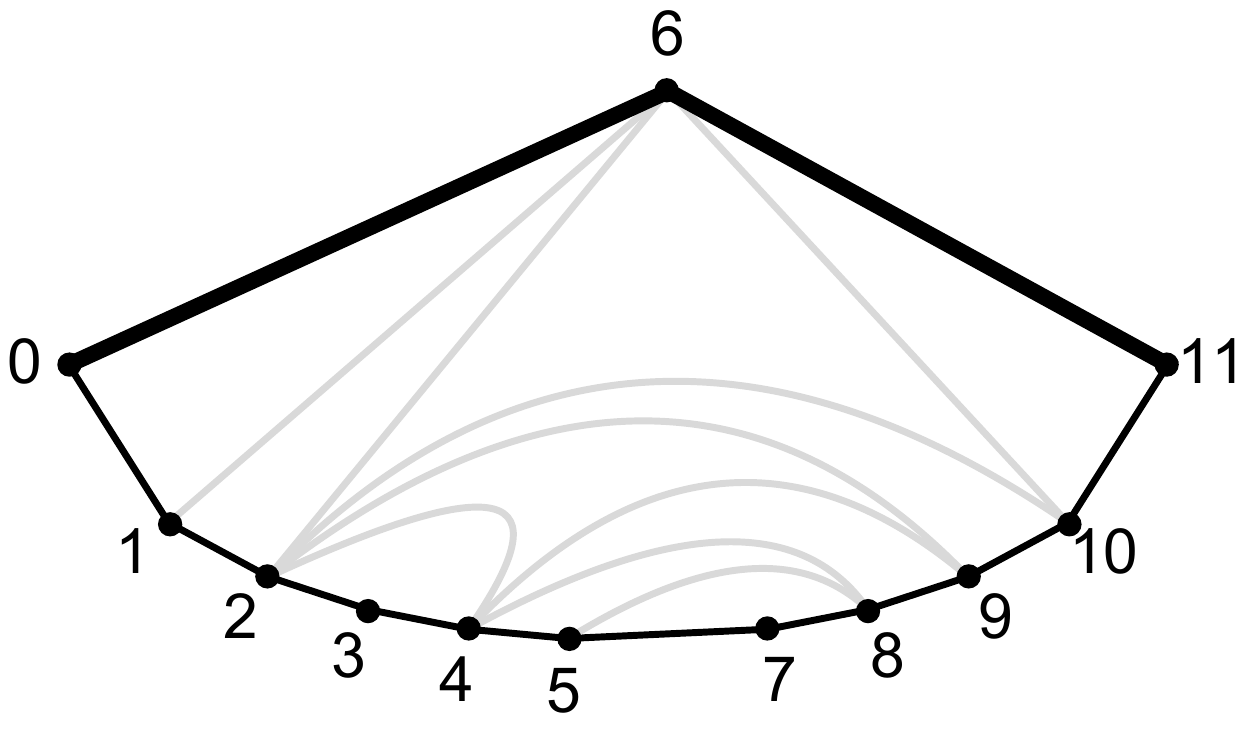}
&\includegraphics[width=.3\textwidth]{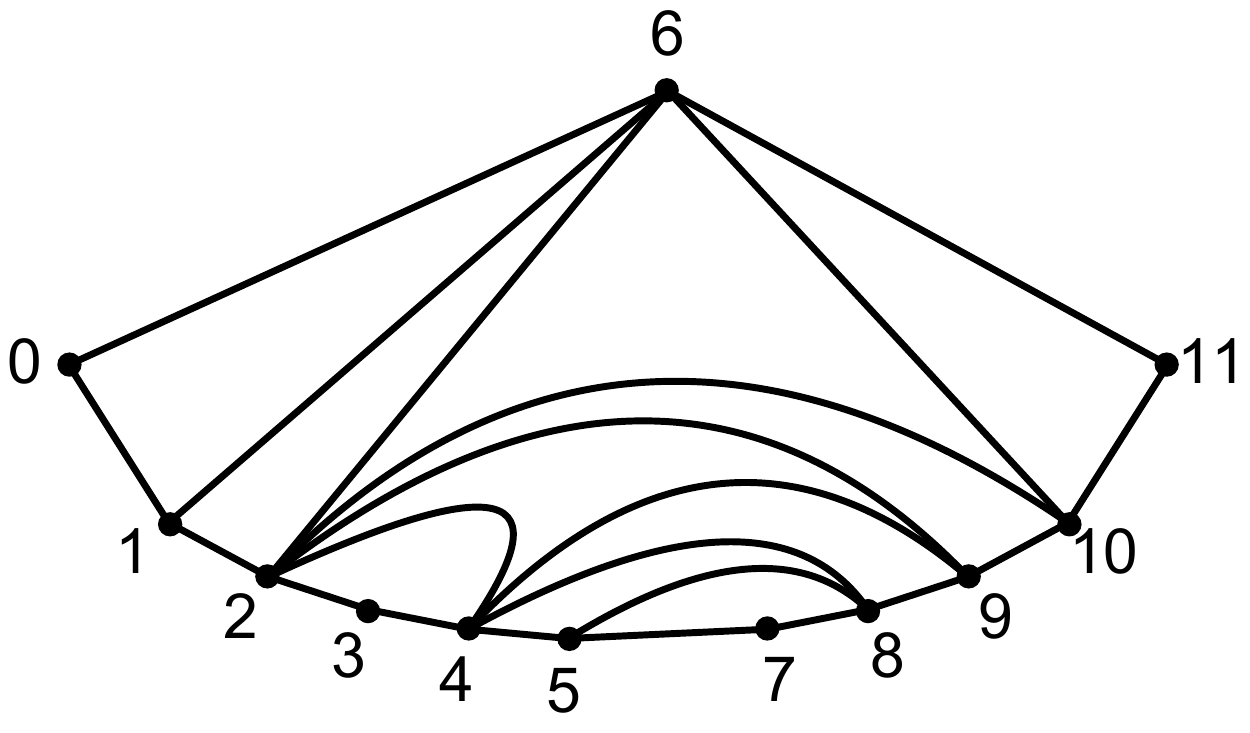}        
\end{array}\]

\caption{An illustration of the bijection $\Tri_c$ between an element $w$ and $\Tri_c(w)$, for $c=s_6s_7s_8s_9s_1s_2s_3s_4s_5$ and $w$ with one-line notation $3,7,5,8,4,9,6,2,1,10$.}
\label{fig:c_sort_to_triangulation}
\end{figure}

We refer the reader to Table 1 of~\cite{ceballos2014subword} for the corresponding map from subwords to triangulations, whereby a unique diagonal of an $(n+2)$-gon is associated to each facet of the associahedron (by Theorem~\ref{prop:parabolic_on_asoc}, this amounts to associating a diagonal to each letter of the word $\mathsf{Q}=\mathsf{cw_o(c)}$).  The important point we require is that---up to a global conventional relabeling of the unconventionally-labeled $(n+2)$-gon---a $c$-sortable element $w$ and its Cambrian rotation are both assigned the same triangulation (note that the description in~\cite{ceballos2014subword} has already imposed this global relabeling).  As a particular case, applying Cambrian rotation in the order given by a fixed reduced word for $c$ gives a map from $\Asoc(A_{n-1},c)$ to $\Asoc(A_{n-1},c)$ that corresponds to simply \emph{rotating} the triangulation.

\subsection{A Combinatorial Projection}

In this section, we give a slightly modified version of D.~Sleator, R.~Tarjan, and W.~Thurston's type $A_n$ combinatorial projection.  

We will label the simple reflections so that $s_i$ has cycle notation $(i,i+1)$.  Let \[c_{i} = s_is_{i+1}\cdots s_{n-1} s_1 s_2 \cdots s_{i-1} = (1,2,\ldots,i-1,i+1,\ldots,n,i),\] so that we may take the single upper vertex of our unconventionally-labeled $(n+2)$-gon corresponding to $c_{i}$ to be the upper vertex $i$.  

Given a $(n+2)$-gon with vertices labeled counterclockwise in the usual way by $0,1,2,\ldots, n+1$, and given a diagonal $f=(f_1,f_2)$ with $f_1<f_2$, there are exactly two ways relabel the vertices of this $(n+2)$-gon so that the labeling corresponds to some $c_{i}$ and so that the new labels of the diagonal $f$ are $i$ and $i+1$.  (One may take either the vertex directly counterclockwise of $f_1$ or directly counterclockwise of $f_2$ to be the vertex whose new label is $0$, and from there label vertices counterclockwise by $1,2,\ldots,i-1,i+1,i+2,\ldots,n+1,i$).  Choose one---this is the choice to project to either the parabolic subgroup $A_i \times A_{n-i-1}$ or to the parabolic subgroup $A_{n-i-1} \times A_{i}$.  That we have a choice at all is, as usual, an artifact of the cute but annoying fact that $w_o s_i w_o = s_{n-i}$ in type $A_{n-1}$.  We note that this choice of relabeling is irrelevant: for either choice, our projection map will define the same triangulation up to rotation.




We now give the combinatorial projection $\projCamb_{i}$.  Fix the diagonal $(i,i+1)$ of the labeled $(n+2)$-gon arising from $c_{i}$, and define $D_{i}$ to be the set of interior diagonals, specified as pairs of vertices.  Given a diagonal $(a,b)$ with $a<b$, we set \[\projCamb_{i}((a,b)):=\begin{cases}\{(a,i+1), (i,b)\} \cap D_{i} & \text{if } a<i \text{ and } b>i\\ \{(a,b)\} & \text{otherwise.} \end{cases}\]  Given a triangulation $T=\{(a_p,b_p) : 1 \leq p \leq n-1\}$, we define \[\projCamb_i(T):= \{(i,i+1)\} \cup \bigcup_p \projCamb_{i}((a_p,b_p)).\]  It is easy to check, using the same reasoning as in~\cite{ceballos2014subword}, that $\projCamb_i$ is a normalization map.  An example is given by the triangulations in Figure~\ref{fig:projections_eq}.

\begin{remark}
We stress here that this differs from the projection in~\cite{sleator1988rotation,ceballos2014subword}, which in our setup would be given by 
\[N_{(i,i+1)}((a,b)):=\begin{cases}\{(a,x), (x,b)\} \cap D_{i} & \text{if } a<i \text{ and } b>i\\ \{(a,b)\} & \text{otherwise,} \end{cases}\] where $x$ is chosen beforehand to (always) be either $i$ or $i+1$.  Our projection has the benefit of eliminating this choice.
\end{remark}





\subsection{Equivalence of Projections}

Using the bijection $\Tri_{c_i}$, we will now show that our modification $\projCamb_{i}$ of D.~Sleator, R.~Tarjan, and W.~Thurston's combinatorial projection  recovers the geometric projection $\projCamb_{W_{\langle i \rangle}}$ defined in Section~\ref{sec:main_theorem}.

We fix $c$ to be the $n$-cycle $c_1=s_1s_2\cdots s_{n-1} = (1,2,\ldots,n)$.  By the discussion in Section~\ref{sec:orient}, any facet of $\Asoc(A_{n-1},c)$ may be rotated to become the $c_{i}$-sortable elements in the standard parabolic subgroup $W_{\langle s_i \rangle}$, for some $i$.  This face is specified by the diagonal $(i,i+1)$.

\begin{theorem}
\label{thm:proj_eq}
	Let $w$ be a $c_{i}$-sortable element.  Then \[\Tri_{c_i}(\projCamb_{W_{\langle s_i \rangle}}(w)) = \projCamb_i(\Tri_{c_i}(w)).\]	
\end{theorem}

An example is given in Figure~\ref{fig:projections_eq}.

\begin{figure}[htbp]
\[\begin{array}{cc}
\includegraphics[width=.45\textwidth]{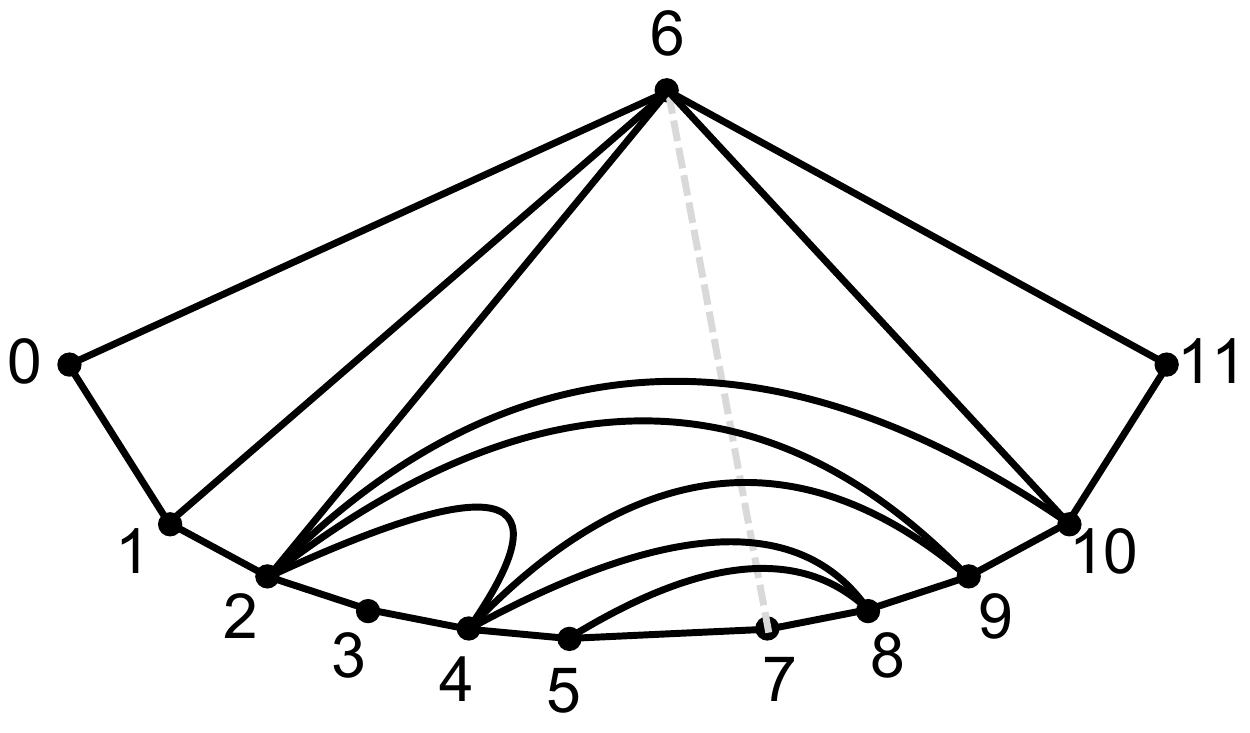} & \includegraphics[width=.45\textwidth]{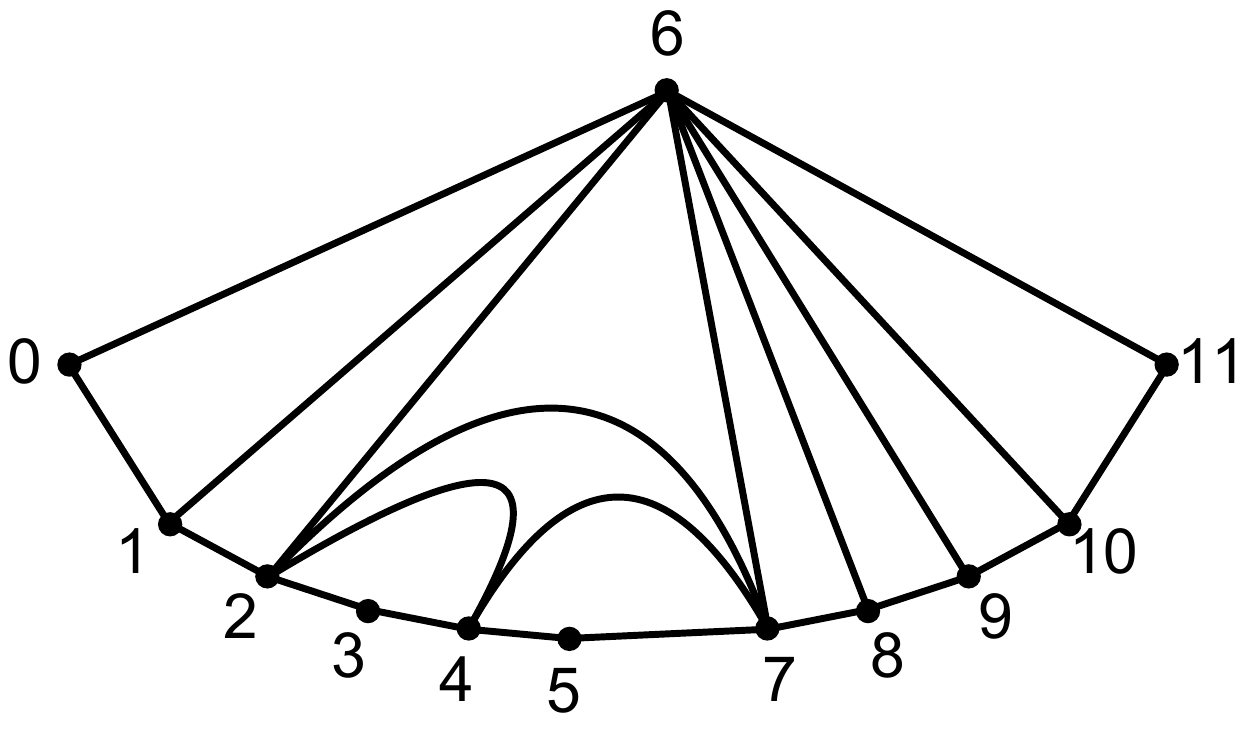} \\
w = 3,7,5,8,4,9,6,2,1,10  & \projCamb_{W_{\langle s_i \rangle}}(w) = 3,5,4,6,2,1 | 7,8,9,10
\end{array}\]

\caption{On the left is a permutation $w$ below its corresponding triangulation.  On the right is the permutation $\projCamb_{W_{\langle s_i \rangle}}(w)$ below its triangulation, illustrating Theorem~\ref{thm:proj_eq}}    
\label{fig:projections_eq}
\end{figure}

\begin{proof}
We first establish the necessary properties of one-line notation, to better understand the map from sortable elements to triangulations.

It is easy to see that the parabolic subgroup $W_{\langle s_i \rangle}$ consists of those permutations whose one-line notation has the property that each letter $1,2,\ldots,i$ is to the left of all letters $i+1,i+2,\ldots,n$.  Then by the bijection between $c_{(i)}$-sortables and triangulations above, the $c_{(i)}$-sortable elements that lie in $W_{\langle s_i \rangle}$ correspond to exactly those triangulations that have the diagonal $(i,i+1)$.


Let $w = w_1 w_2 \cdots w_n$ be the one-line notation of $w$, let $w_{j_1}w_{j_2}\cdots w_{j_i}$ be its (in-order) restriction to those $w_j \leq i$, and let $w_{j_{i+1}}w_{j_{i+2}}\cdots w_{j_{n}}$ be its restriction to those $w_k > i$.  Recall that \[\inv(\projCamb_{W_{\langle s_i \rangle}}(w))) = \inv(w) \cap \inv\left((w_o)_{\langle s_i \rangle}\right).\]   Then $u=\projCamb_{W_{\langle s_i \rangle}}(w)))$ has one-line notation $u_1 u_2 \cdots u_n$, where the first $i$ letters are the numbers from $1$ to $i$ and their relative ordering in $w$ is preserved, and similarly for the last $n+1-i$ letters.  Precisely, \[u_1 u_2 \cdots u_i = w_{j_1} w_{j_2} \cdots w_{j_i} \text{ and } u_{i+1} u_{i+2} \cdots u_n = w_{j_{i+1}}w_{j_{i+2}}\cdots w_{j_{n}}.\]

The theorem now follows immediately by noting that the lower vertex $i+1$ is contained in each of the paths $P_k$ for $1 \leq k \leq i$, while the upper vertex $i$ is contained in each of the paths $P_k$ for $i+1 \leq k \leq n$.  Furthermore, the above considerations make it clear that this equality actually holds for any element of $A_{n-1}$.
\end{proof}






\section{Final Remarks}

It would be interesting to compare the combinatorial projection maps defined for types $B$ and $D$ in~\cite{ceballos2014diameter} with the combinatorics one could extract from our geometric projection.  One would also expect that such a result holds for the Fuss-analogues of associahedra (despite the absence of hyperplane geometry in this situation, the projection map would be defined by projecting onto what passes for a standard parabolic in the positive Artin monoid).  Finally, we note that the similarity of the proofs for the $W$-permutahedron and the $W$-associahedron suggest that the proper generality for an in-your-face theory has not been found.

\bibliographystyle{amsalpha}
\bibliography{inyourface}
\end{document}